\documentclass[11pt,reqno,british,a4paper,twoside]{article}

\usepackage[utf8]{inputenc}
\usepackage[T1]{fontenc}
\usepackage{mathpazo}
\usepackage{tgpagella}

\usepackage{babel,mathtools,amssymb,xspace,mathrsfs}
\usepackage{booktabs,array,fixltx2e,fancyhdr,csquotes,stmaryrd,textcomp}
\usepackage[inline,shortlabels]{enumitem}

\usepackage[babel=true]{microtype}
\usepackage[margin=30pt,font=small,labelfont=bf,labelsep=period]{caption}

\usepackage{tikz}
\usetikzlibrary{matrix}

\usepackage[amsmath,thmmarks]{ntheorem}
\usepackage[hyperfootnotes=false,hidelinks]{hyperref}

\rmfamily
\mathtoolsset{mathic=true}

\setlist[enumerate,1]{label=\upshape (\roman*),nosep}{}{}{}
\newlist{asparaenum}{enumerate}{1}
\setlist[asparaenum,1]{labelindent=\parindent,leftmargin=0pt,
itemindent=2\parindent,label=\upshape (\roman*)}

\theoremstyle{plain}
\theoremseparator{.}
\newtheorem{theorem}{Theorem}[section]
\newtheorem{proposition}[theorem]{Proposition}
\newtheorem{lemma}[theorem]{Lemma}
\newtheorem{corollary}[theorem]{Corollary}

\theorembodyfont{\normalfont}
\newtheorem{definition}[theorem]{Definition}

\theoremheaderfont{\itshape}
\theoremsymbol{{\small\ensuremath{\quad\triangle}}}
\newtheorem{remark}[theorem]{Remark}
\theoremsymbol{{\small\ensuremath{\quad\diamondsuit}}}

\theoremstyle{nonumberplain}
\theoremheaderfont{\itshape}
\theorembodyfont{\normalfont}
\theoremsymbol{{\small\ensuremath{\quad\Box}}}
\newtheorem{proof}{Proof}

\qedsymbol{{\small\ensuremath{\quad\Box}}}

\numberwithin{equation}{section}

\newcommand{\hor}{{\mathscr H}}
\newcommand{\Hrel}{\mathrel{\sim_{\mkern-8mu\hor}^{}}}

\newcommand{\Hrelated}{\( \hor \)\hyphen related\xspace}

\newcommand{\Lie}[1]{\operatorname{\mathit{#1}}}
\newcommand{\lie}[1]{\operatorname{\mathfrak{#1}}}

\newcommand{\GL}{\Lie{GL}}
\newcommand{\gl}{\lie{gl}}
\newcommand{\ort}{\lie{o}}
\newcommand{\SO}{\Lie{SO}}
\newcommand{\so}{\lie{so}}

\newcommand{\Sp}{\Lie{Sp}}
\newcommand{\sP}{\lie{sp}}

\newcommand{\Un}{\Lie{U}}
\newcommand{\un}{\lie{u}}

\newcommand{\bC}{{\mathbb C}}
\newcommand{\bH}{{\mathbb H}}
\newcommand{\bR}{{\mathbb R}}
\newcommand{\bZ}{{\mathbb Z}}

\newcommand{\ii}{{\mathbf i}}
\newcommand{\jj}{{\mathbf j}}
\newcommand{\Gi}{{\mathbf s}}
\newcommand{\G}{{\mathbf G}}

\newcommand{\tJ}{{\widetilde J}}
\newcommand{\tK}{{\widetilde K}}

\DeclareMathOperator{\CH}{\bC H}

\DeclareMathOperator{\End}{End}
\DeclareMathOperator{\Id}{Id}
\DeclareMathOperator{\im}{Im}
\DeclareMathOperator{\re}{Re}
\DeclareMathOperator{\RH}{\bR H}

\DeclareMathOperator{\diag}{diag}
\DeclareMathOperator{\vol}{vol}

\newcommand{\UM}{\mathscr U}

\newcommand{\symp}{{\sslash}}
\newcommand{\hkq}{{/\mkern-6mu/\mkern-6mu/}}

\newcommand{\hmod}{{\mathrm{mod}}} 

\newcommand{\hook}{{\lrcorner\,}}

\newcommand{\NB}{\nabla}
\newcommand{\NC}{\nabla^{\mathrm{C}}}
\newcommand{\LC}{\NB^{\mathrm{LC}}}
\newcommand{\oLC}{\omega_{\mathrm{LC}}}
\newcommand{\OLC}{\Omega_{\mathrm{LC}}}
\newcommand{\nA}{\eta_{\mathrm{A}}}
\newcommand{\oC}{\omega_{\mathrm{C}}}
\newcommand{\oN}{\omega_\nabla}
\newcommand{\ON}{\Omega_\nabla}

\newcommand{\hX}{\widetilde X}

\newcommand{\oma}[1]{\omega^\alpha_#1}
\newcommand{\omam}[1]{\overline{\omega^\alpha_#1}}
\newcommand{\omb}[1]{\omega^\beta_#1}

\DeclarePairedDelimiter{\degr}{\lvert}{\rvert}
\DeclarePairedDelimiter{\norm}{\lVert}{\rVert}
\DeclarePairedDelimiterX{\inp}[2]{\langle}{\rangle}{#1, #2}
\DeclarePairedDelimiter{\Span}{\langle}{\rangle}

\DeclarePairedDelimiterX{\Set}[2]{\{}{\}}{\, #1 \,\delimsize\vert\, #2 \,}

\newcommand{\eqbreak}[1][2]{\\&\hskip#1em}

\providerobustcmd*{\hyphen}{%
  \nobreak-\nobreak\hskip0pt}

\begin{document}

\thispagestyle{empty}
\begin{flushright}
    \small
\end{flushright}

\bigskip

\begin{center}
  \LARGE\bfseries Twist geometry of the c-map
\end{center}
\begin{center}
  \Large Oscar Macia and Andrew Swann
\end{center}

\begin{abstract}
  We discuss the geometry of the c-map from projective special Kähler
  to quaternionic Kähler manifolds using the twist construction to
  provide a global approach to Hitchin's description.  As found by
  Alexandrov et al.\ and Alekseevsky et al.\ this is related to the
  quaternionic flip of Haydys.  We prove uniqueness statements for
  several steps of the construction.  In particular, we show that
  given a hyperKähler manifold with a rotating symmetry, there is
  essentially only a one parameter degree of freedom in constructing a
  quaternionic Kähler manifold of the same dimension.  We demonstrate
  how examples on group manifolds arise from this picture.
\end{abstract}

\bigskip
\begin{center}
  \begin{minipage}{0.8\linewidth}
    \microtypesetup{protrusion=false}
    \hrule \vspace{\baselineskip}
    {\small \tableofcontents\par}
    \vspace{1.5\baselineskip}\hrule
  \end{minipage}
\end{center}



\section{Introduction}
\label{sec:introduction}

The c-map was introduced in the physics literature by Cecotti, Ferrara
and Girardello \cite{Cecotti-FG:II} and explicit local expressions for
the metrics were provided by Ferrara and Sabharwal \cite{Ferrara-S:q}.
It is a remarkable construction that for certain Kähler manifolds~\( S
\) of dimension~\( 2n \) creates a quaternionic Kähler manifold~\( Q
\) of dimension \( 4n+4 \).  The manifolds \( Q \) are Einstein with
negative scalar curvature and have holonomy contained in the group \(
\Sp(n+1)\Sp(1) \).  The map is derived from a duality of moduli spaces
for type IIA and IIB string theories which takes a product \( S \times
Q' \times \CH(1) \) to a product \( S' \times Q \times \CH(1) \), with
\( Q \) the c-map of \( S \) and \( Q' \) the c-map of \( S' \).

Historically the c-map construction has had particular importance for
the study of homogeneous quaternionic Kähler manifolds.  In
\cite{Cecotti-FG:II}, Alekseevsky's method~\cite{Alekseevsky:solvable}
for classifying completely solvable Lie groups admitting a
quaternionic Kähler metric was used to elucidate the geometric
properties of the c-map for homogeneous spaces.  Later the c-map was
used by de Wit and Van Proeyen \cite{DeWit-vP:special} to show that
the above homogeneous classification was missing one family of spaces,
a derivation of the correct result using Alekseevsky's construction
was then provided by Cortés~\cite{Cortes:Alekseevskian}.  The
significance of these spaces is that they include the only known
examples of homogeneous quaternionic Kähler manifolds that are not
symmetric.

Given the importance of these homogeneous results, it is surprising
that the first mathematical description of the c-map was published by
Hitchin~\cite{Hitchin:quaternionic} in~2009, building on the
description of the relevant Kähler geometries in
Freed~\cite{Freed:special}.  As was well-known, an intermediate step
is the construction of a hyperKähler manifold~\( H \), which is the
rigid c-map of a cone~\( C \) on~\( S \).  It then became apparent
that the passage from the hyperKähler manifold~\( H \) to the
quaternionic Kähler~\( Q \), was a pseudo-Riemannian version of a much
more general construction of Haydys~\cite{Haydys:hKS1}.  This has now
been dubbed the hK/qK correspondence and has been studied both in the
physics literature, for example \cite{Alexandrov-DP:qK-hK}, and
mathematically, for example
\cite{Hitchin:hK-qK,Hitchin:hyperholomorphic}.  In the particular
context of the c-map, Cortés et al.\
\cite{Alekseevsky-CDM:qK-special,Alekseevsky-CM:conification,
Cortes-HX:supergravity} have shown that this correspondence reproduces
the explicit local expressions for the c-map and its one-loop
deformation, first seen from a twistor viewpoint by Alexandrov et al.\
\cite{Alexandrov-DP:qK-hK}, and used it to construct new examples of
complete inhomogeneous quaternionic Kähler manifolds.

Given that the c-map is a manifestation of a general string theory
duality, it is useful to put the above constructions into a wider
context.  In \cite{Swann:T,Swann:twist} a twist construction was
introduced as a geometric interpretation of a broad class of T-duality
constructions.  It was successfully exploited to produce geometries
with torsion and particular examples of hypercomplex structures.
However, it was not apparent how it could produce Riemannian metrics
of special holonomy.  The central purpose of this paper is to show how
to adapt the twist construction so that hyperKähler manifolds with a
rotating circle symmetry may be used to produce quaternionic Kähler
manifolds.  The method to do this involves considering a combination
of a conformal change in the hyperKähler metric together with a
different conformal scaling along quaternionic directions associated
to the symmetry.  We prove that there is essentially only one degree
of freedom if the result is to be quaternionic Kähler.  It immediately
follows that the descriptions of the hK/qK correspondence in
Haydys~\cite{Haydys:hKS1}, Hitchin~\cite{Hitchin:hK-qK} and
Alekseevsky, Cortés and Mohaupt \cite{Alekseevsky-CDM:qK-special}
agree and that the twist construction may be used to describe the
c-map.  We will therefore present the construction of the c-map in
this context, with emphasis on proving uniqueness of the constructions
involved.  We take a uniformly global approach, with the aim of
elucidating the geometric basis for the c-map and illustrating how
information may be extracted from the twist construction.  We consider
metrics of arbitrary signature when appropriate.  In
\cite{Swann:twist-mod} similar constructions are considered for
tri-holomorphic actions on hyperKähler manifolds; in contrast to the
single parameter in the present paper it turns that there is much more
freedom in the way of obtaining twists that are again hyperKähler.

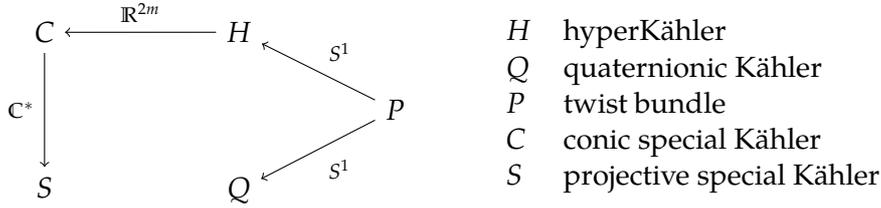
\begin{figure}[tp]
  \begin{tikzpicture}[baseline=0pt,->]
    \matrix[matrix of math nodes,row sep=0.5cm] {
    |(C)| C &[2cm] |(H)| H &[1.5cm]         \\
            &         & |(P)| P \\
    |(S)| S & |(Q)| Q &         \\};
    \begin{scope}[every node/.style={midway,font=\scriptsize}]
      \draw (C) to node[left] {\( \bC^* \)} (S);
      \draw (H) to node[above] {\( \bR^{2m} \)} (C);
      \draw (P) to node[above right] {\( S^1 \)} (H);
      \draw (P) to node[below right] {\( S^1 \)} (Q);
    \end{scope}
  \end{tikzpicture}
  \qquad
  \begin{tabular}[c]{ll}
    \( H \) & hyperKähler          \\
    \( Q \) & quaternionic Kähler  \\
    \( P \) & twist bundle         \\
    \( C \) & conic special Kähler \\
    \( S \) & projective special Kähler
  \end{tabular}
  \caption{Spaces in the c-map}
  \label{fig:spaces}
\end{figure}

In outline the constructions are as follows,
see~Figure~\ref{fig:spaces}.  A projective special Kähler manifold \(
S \) may be defined via a complex reduction of a regular conic special
Kähler manifold~\( C \).  For supergravity the latter has complex
Lorentzian signature.  The rigid c-map defines an indefinite
hyperKähler metric on \( H = T^*C \).  This carries a circle action
and the aim is to produce a quaternionic Kähler manifold \( Q \) of
the same dimension as \( H \).  Hitchin provides a local description,
using the flat special Kähler connection to write \( T^*C = C \times
\bR^{2m} \), putting \( Q = S \times \CH(m+1) \) and adjusting the
metric along quaternionic directions related to the symmetry.
Patching of Hitchin's local construction is described by Cortés, Han
and Mohaupt \cite{Cortes-HX:supergravity}.  However, we will show that
a global picture may be obtained by considering Haydys' quaternionic
flip construction in indefinite signature, as also found by Alexandrov
et al.\ \cite{Alexandrov-DP:qK-hK} and Alekseevsky et al.\
\cite{Alekseevsky-CDM:qK-special}, and that in addition the
constructions may be directly described via the twist construction
of~\cite{Swann:twist}.  The twist construction lifts the circle action
on \( H \) to a principal circle bundle \( P \) and constructs \( Q \)
as the circle quotient of \( P \) by the lifted action.  The main
conclusions are that the c-map for projective special Kähler manifolds
is obtained from the rigid c-map, via general constructions for
hyperKähler manifolds with a rotating circle symmetry, and that these
latter constructions are essentially unique.

\paragraph*{Acknowledgements}
\label{sec:acknowledgements}

We thank Thomas Bruun Madsen, Nigel Hitchin, Vicente Cortés, Christoph
Böhm, Paul Gauduchon and Henrik Pedersen for useful conversations, and
Maxim Kontsevich for an instigating question.  A preliminary version
of this material was presented at the Srní Winter School, 2011, we
thank the organisers for that opportunity.  This work is partially
supported by the Danish Council for Independent Research, Natural
Sciences and by the Spanish Agency for Scientific and Technical
Research (DGICT) and FEDER project MTM2010\hyphen 15444.

\section{The rigid c-map}
\label{sec:rigid-c-map}

The rigid c-map constructs hyperKähler manifolds of dimension~\( 4n \)
from so-called special Kähler manifolds of dimension~\( 2n \).  It was
described mathematically by Freed~\cite{Freed:special}.  Let us
demonstrate this construction using the language of structure bundles
and show how the special Kähler condition arises naturally.  For later
use we will need the case of indefinite metrics.

\begin{definition}
  A \emph{special Kähler} manifold is a Kähler manifold \(
  (M,g,I,\omega) \), with \( g \) a metric of signature \( (2p,2q) \),
  together with a flat, torsion-free, symplectic connection~\( \NB \)
  satisfying \( d^\NB I = 0 \), meaning
  \begin{equation}
    \label{eq:special}
    (\NB_XI)Y = (\NB_YI)X,
  \end{equation}
  for all \( X,Y \in TM \).
\end{definition}

Our conventions are such that \( \omega_I(\cdot,\cdot) =
g(I\cdot,\cdot) \).  When \( M \) is a Kähler manifold, there are
sophisticated constructions of hyperKähler metrics on neighbourhoods
of the zero section of \( T^*M \) due to Feix~\cite{Feix:cotangent} and
Kaledin~\cite{Kaledin:cotangent}.  We will show how the extra conditions
of a special Kähler geometry arise naturally from the desire to have a
simply defined hyperKähler metric on the whole cotangent bundle.

\subsection{HyperKähler structures on flat vector spaces}
\label{sec:hyperk-struct-flat}

Let \( V \) denote \( \bC^{p,q} \) regarded as \( (\bR^{2m},\G,\ii)
\), \( m = p + q \), so that the (indefinite) inner product is \(
\inp{\mathbf v}{\mathbf w} = \mathbf v^T\G\mathbf w \) and \( \ii \)
is the complex structure \( \ii\G = \G\ii \), \( \ii^T = - \ii \).
Concretely, take \( \G = \diag(\Id_{2p},-\Id_{2q}) \) and \( \ii =
\diag(\ii_2,\dots,\ii_2) \), where \( \ii_2 =
\begin{psmallmatrix}
  0 & -1 \\
  1 & 0
\end{psmallmatrix}
\).  The vector space \( H = V + V^* \) is isomorphic to \( \bH^{p,q}
\) and the flat (indefinite) hyperKähler metric may be described as
follows.  Let \( \theta \in \Omega^1(V,\bR^{2m}) \) be the
tautological one-form from the identification \( T_xV \cong V =
\bR^{2m} \).  In other words,
\begin{equation*}
  \theta = (dx_1,dy_1,dx_2,dy_2,\dots,dx_m,dy_m)^T,
\end{equation*}
where \( (x_1,y_1,\dots) \) are the standard coordinates on \(
\bR^{2m} \) and the flat metric on \( V \) is
\begin{equation*}
  g = \sum_{i=1}^p dx_i^2 + dy_i^2 - \sum_{j=p+1}^m dx_j^2 + dy_j^2.
\end{equation*}
The Kähler form on \( V \) is
\begin{equation*}
  \omega =  \sum_{i=1}^p dx_i \wedge dy_i - \sum_{j=p+1}^m dx_j \wedge dy_j =
  - \tfrac12 \theta^T \wedge \Gi \theta,
\end{equation*}
with
\begin{equation*}
  \Gi = \G \ii.
\end{equation*}
Similarly on \( V^* \), let \( \alpha \in \Omega^1(V^*,{\bR^{2m}}^*)
\) be the tautological form:
\begin{equation*}
  \alpha = (du_1,dv_1,du_2,dv_2,\dots,du_m,dv_m),
\end{equation*}
with \( (u_1,v_1,\dots) \) dual coordinates to \( (x_1,y_1,\dots) \).
An induced Kähler form on \( V^* \) is
\begin{equation*}
  \omega^* =   \sum_{i=1}^p du_i\wedge dv_i - \sum_{j=p+1}^m
  du_j\wedge dv_j = - \tfrac12 \alpha \wedge \Gi \alpha^T .
\end{equation*}

Now regarding \( H = V+V^* \) as \( T^*V \) there is a canonical
symplectic form
\begin{equation*}
  \omega_J = \alpha \wedge \theta = \sum_{i=1}^m du_i\wedge dx_i +
  dv_i\wedge dy_i.
\end{equation*}
In addition, using \( \ii \) we may construct
\begin{equation*}
  \omega_K = - \alpha \wedge \ii \theta = \sum_{i=1}^m du_i \wedge dy_i
  - dv_i \wedge dx_i.
\end{equation*}
Together with
\begin{equation*}
  \omega_I = \omega - \omega^* = \tfrac12 ( \alpha \wedge \Gi
  \alpha^T - \theta^T \wedge \Gi \theta ) 
\end{equation*}
we obtain a flat hyperKähler structure with metric \( 
\sum_{i=1}^m \varepsilon_i(dx_i^2 + dy_i^2 + du_i^2 + dv_i^2) \),
where \( \varepsilon_i = \G_{ii} = \pm1 \), and complex structures:
\begin{gather*}
  Idx_i = dy_i,\quad Idu_i = - dv_i, \qquad
  Jdv_i = dy_i,\quad Jdx_i = -du_i, \\
  Kdu_i = dy_i,\quad Kdv_i = -dx_i.
\end{gather*}

\subsection{The cotangent bundle}
\label{sec:cotangent-bundle}

Let us start with a general manifold \( M \) of dimension \( n \).
The bundle \( \GL(M) \) of frames, i.e., linear isomorphisms \(
u\colon\bR^n \to T_aM \), is a principal \( \GL(n,\bR) \)-bundle with
action \( (R_gu)(v) = (u\cdot g)(v) = u(gv) \), for \( g \in
\GL(n,\bR) \) and \( v \in \bR^n \).  It carries a canonical one-form
\( \theta \in \Omega^1(\GL(M),\bR^n) \) given by \( \theta_u(X) =
u^{-1}(\pi_*X) \), where \( \pi\colon \GL(M) \to M \) is the
projection.  This satisfies \( R_g^*\theta = g^{-1}\theta \).  A
connection one-form \( \oN \in \Omega^1(\GL(M),\End(\bR^n)) \) is by
definition a form such that \( R_g^*\oN = g^{-1}\oN g \) and \(
\oN(\xi^*) = \xi \), where \( \xi^* \) is the vector field on \(
\GL(M) \) generated by the infinitesimal action of \( \xi
\in \End(\bR^n) \), the Lie algebra of \( \GL(n,\bR) \).  The
connection is torsion-free if and only if
\begin{equation*}
  d\theta = - \oN \wedge \theta.
\end{equation*}

The cotangent bundle may be constructed as the associated bundle
\begin{gather*}
  \GL(M) \times_{\GL(n,\bR)} (\bR^n)^* = T^*M,\\
  R_g(u,v) = (u\cdot g,vg)\mapsto v\circ u^{-1}.
\end{gather*}
Writing \( x\colon (\bR^n)^* \to (\bR^n)^* \) for the identity map, we
have \( R_g^*x = xg \).  The form
\begin{equation*}
  \alpha = dx - x\oN
\end{equation*}
on \( \GL(M) \times (\bR^n)^* \) agrees with \( dx \) on \( (\bR^n)^*
\) and is zero on vectors tangent to the \( \GL(n,\bR) \)-action.  It
satisfies \( R_g^*\alpha = \alpha g \).  As the kernels of the forms
\( \alpha \) and \( \theta \) are preserved by the group action, these
kernels descend to provide a splitting
\begin{equation}
  \label{eq:V-H}
  T(T^*M) = \mathcal V \oplus \mathcal H,
\end{equation}
with \( \mathcal V = \ker \pi_* \) and \( \mathcal H_{[u,v]} \cong T_{\pi(u)}M \).

The canonical symplectic form on \( T^*M \) is now
\begin{equation*}
  \omega_J = d(x\theta),
\end{equation*}
since \( x\theta \) is the tautological one-form.  Expanding the
right-hand side gives \( \omega_J = dx \wedge \theta + x d\theta \).
This gives
\begin{lemma}
  \( \omega_J = \alpha\wedge\theta \) if and only if \( \oN \) is
  torsion-free.
  \qed
\end{lemma}

Now suppose that \( M \) carries an almost complex structure \( I \)
and that \( n = 2m \).  Consider the bundle \( \GL(\bC,M) \) of \( \bC
\)-linear frames \( u\colon \bC^m=\bR^{2m} \to T_aM \), \( u \circ \ii
= I\circ u\).  If we identify the cotangent bundle with \(
\Lambda^{1,0}M = \GL(\bC,M) \times_{\GL(m,\bC)} (\bC^m)^* \), then we
obtain a canonical non-degenerate closed \( (2,0) \)-form.  Its real
part is the canonical symplectic structure \( \omega_J \) given above,
where \( \theta \) is pull-backed to \( \GL(\bC,M) \).  The imaginary
part is
\begin{equation*}
  \omega_K = -d(x\ii \theta).
\end{equation*}
This expands to \( \omega_K = - dx\ii\wedge \theta - x\ii d\theta =
- (dx\ii - x\ii\oN) \wedge \theta \).  This gives

\begin{lemma}
  Suppose \( \oN \) is torsion free.  We have \( \omega_K = - \alpha
  \wedge \ii \theta \) if and only if the complex structure \( I \) on
  \( M \) is integrable and the pair \( (\NB,I) \) satisfies the
  special condition~\eqref{eq:special}.
\end{lemma}

\begin{proof}
  The expansion of \( \omega_K \) above shows that \( \omega_K = -
  \alpha \wedge \ii \theta \) if and only if
  \begin{equation}
    \label{eq:oK-oN}
    (\ii\oN - \oN\ii) \wedge \theta = 0,
  \end{equation}
  where \( \oN \) is pulled-back to \( \GL(\bC,M) \).  Writing \( \oN
  = \oC + \nA \), where \( \oC = \tfrac12(\oN - \ii \oN \ii) \) and \(
  \nA = \tfrac12(\oN + \ii \oN \ii) \) are the complex and
  anti-complex parts of \( \oN \), we have that \( \oC \) takes values
  in \( \gl(m,\bC) \).  Equation~\eqref{eq:oK-oN} is equivalent to
  \begin{equation}
    \label{eq:oA}
    \nA \wedge \theta = 0.
  \end{equation}
  As \( \oN \) is torsion-free, we have \( d\theta = -\oC \wedge
  \theta \) and so \( \oC \) is a torsion-free \( \GL(m,\bC) \)
  connection.  Thus we have a torsion-free connection \( \NC \) such
  that \( \NC I = 0 \).  By the Newlander-Nirenberg Theorem, such a
  connection exists if and only if \( I \) is integrable.

  Now equation~\eqref{eq:oA} is equivalent to \( 0 = \ii\nA \wedge
  \theta = - \nA \wedge \ii\theta \).  This is zero on vertical
  vectors, and for any \( X,Y\in TM \) taking any lifts \(
  X',Y',(IX)',(IY)' \) to \( T\Un(M) \) we have
  \begin{equation*}
    \ii \theta(X') = \theta((IX)') \quad\text{and}\quad
    0 = \nA(X')\theta((IY)') - \nA(Y')\theta((IX)'), 
  \end{equation*}
  giving \( (\NC-\NB)_X(IY) = (\NC-\NB)_Y(IX) \).  As \( \NC I = 0 \)
  and the two connections \( \NC \) and \( \NB \) are torsion-free,
  this is equivalent to~\eqref{eq:special}.
\end{proof}

Let us now assume that \( (M,I) \) has a Hermitian metric \( g \),
possibly of indefinite signature.  Then the structure group reduces to
\( \Un(p,q) \).  Let \( \Un(M) \) denote the bundle of unitary frames.
Pulling \( \theta \) and \( \oN \) back to \( \Un(M) \), these forms
satisfy the identities given on \( GL(M) \), in particular \(
\oN(\xi^*) = \xi \) for each \( \xi \in \un(p,q) \subset \ort(2p,2q)
\subset \End(\bR^n) \).

Considering the flat model we see that the Hermitian form on \( M \)
pulls-back to
\begin{equation*}
  - \tfrac12 \theta^T \wedge \Gi\theta.
\end{equation*}
Indeed the identity \( \pi^*g = \theta^T\G\theta \) implies \(
\pi^*\omega(X,Y) = \pi^*g(IX,Y) = \theta(IX)^T\G\theta(Y) = (\ii
\theta(X))^T\G\theta(Y) = - \theta(X)^T\Gi\theta(Y) = -
\tfrac12(\theta^T \wedge \Gi\theta)(X,Y) \).

On \( T^*M \), it is natural to look for a hyperKähler metric whose
Kähler form for~\( I \) is
\begin{equation}
  \label{eq:oI}
  \omega_I = \tfrac12 (\alpha \wedge \Gi\alpha^T - \theta^T \wedge
  \Gi\theta ). 
\end{equation}
Since such a hyperKähler metric pulls-back to the zero section as the
given Hermitian structure on~\( M \), we see that \( M \) is
necessarily Kähler.  In particular, the connection one-form \( \oLC \)
for the Levi-Civita connection on~\( M \) is a \( \Un(p,q) \)-connection
and torsion-free, so
\begin{equation*}
  \oLC^T\G = - \G\oLC,\quad \oLC \ii = \ii \oLC \quad\text{and}\quad
  d\theta = - \oLC \wedge \theta.
\end{equation*}

\begin{proposition}
  The two-forms \( \omega_I \) of~\eqref{eq:oI}, \( \omega_J = \alpha
  \wedge \theta \), \( \omega_K = - \alpha \wedge \ii \theta \) on \(
  T^*M \) give a hyperKähler structure compatible with the standard
  complex symplectic structure if and only if \( (M,I,g,\nabla) \) is
  special Kähler.
\end{proposition}

The passage from \( (M,I,g,\nabla) \) to the above hyperKähler
structure on~\( T^*M \) is known as the \emph{rigid c-map}.

\begin{proof}
  It remains to show that for a torsion-free connection~\( \NB \),
  closure of \( \omega_I \) corresponds to \( \NB \) being flat and
  symplectic.  We compute
  \begin{equation*}
    2d\omega_I = - d\alpha \wedge \Gi\alpha^T + \alpha \Gi \wedge
    d\alpha^T. 
  \end{equation*}
  We have \( d\alpha = - dx \wedge \oN - x d\oN = - \alpha \wedge \oN
  - x \ON \), where \( \ON = d\oN + \oN \wedge \oN \) is the curvature
  of \( \nabla \).  This gives
  \begin{equation}
    \label{eq:d-om-I}
    \begin{split}
      2d\omega_I &= \alpha \wedge \oN \Gi \wedge \alpha^T + x \ON
      \Gi \wedge \alpha^T \eqbreak + \alpha \wedge \Gi\oN^T \wedge
      \alpha^T - \alpha \wedge \Gi\ON^T x^T\\
      &= 2(\alpha \wedge \oN \Gi \wedge \alpha^T + x \ON \Gi
      \wedge \alpha^T),
    \end{split}
  \end{equation}
  where we have used \( (\beta \wedge \gamma)^T =
  (-1)^{\degr\beta\degr\gamma}\gamma^T \wedge \beta^T \), \( \Gi^T =
  -\Gi \) and noted that each summand~\( \sigma \) of \( d\omega_I \)
  takes values in scalars, so satisfies \( \sigma^T = \sigma \).
  Evaluating \eqref{eq:d-om-I} on \( X,Y,Z \) with \( X,Y \in T\Un(M)
  \), tangent to the principal frame bundle~\( \Un(M) \), and \( Z \in
  T({\bR^{2m}}^*) \), we see that \( d\omega_I = 0 \) implies \( \ON =
  0 \), i.e., \( \nabla \) is flat.  Evaluation on \( X \in T\Un(M) \)
  and \( Y,Z \in T({\bR^{2m}}^*) \), gives that for \( \nabla \) flat
  \( d\omega_I = 0 \) is equivalent to
  \begin{equation}
    \label{eq:oN-G}
    \oN \Gi + \Gi \oN^T = 0.
  \end{equation}
  This says that \( \oN \) is symplectic, since
  \begin{equation*}
    \sP(2m,\bR) \cong \{\, A \in M_n(\bR) : A^T\jj + \jj A = 0 \,\},
  \end{equation*}
  for any \( \jj \) with \( \jj^2 = -1 \), and in particular for \(
  \jj = \Gi \).  
\end{proof}

For future use we note that \( \oN \) satisfies
\begin{equation}
  \label{eq:oN-G-th}
  (\G\oN + \oN^T\G) \wedge \theta = 0,
\end{equation}
which follows from \eqref{eq:oK-oN} and~\eqref{eq:oN-G}.

Now that we have a flat symplectic connection~\( \NB \), it is
reasonable to write out the above structures in adapted coordinates.
Suppose \( s \) is a flat symplectic frame over on open subset~\( M_0
\) of~\( M \), i.e., a section \( M_0 \to \Sp(M_0) \subset \GL(M_0) \)
of the bundle of symplectic frames.  We have \( s^*\oN = 0 \), \(
\sigma_a \coloneqq (s^*\theta)_a = s(a)^{-1} \colon T_aM \to \bR^{2m}
\).  Thus writing \( \tilde s \) for the map
\begin{equation}
  \label{eq:tilde-s}
  \tilde s \coloneqq s\times\Id\colon M_0 \times (\bR^{2m})^* \to
  \GL(M_0) \times (\bR^{2m})^*, 
\end{equation}
gives
\begin{equation}
  \label{eq:local}
  \begin{gathered}
    \tilde s^*(\omega_J) = dx\wedge \sigma,\quad
    \tilde s^*(\omega_K) = - dx \wedge h^{-1}\ii h\, \sigma,\\
    \tilde s^*(\omega_I) = \tfrac12(dx \wedge \Gi dx^T - \sigma^T
    \wedge\Gi\sigma),
  \end{gathered}
\end{equation}
where \( s = R_hu \) for any local unitary frame~\( u \).  Note that
\( h \) is a function on \( M_0 \), so \( x \) only enters the above
expressions through its differential.  Thus translations \( T_v(x) = x
+ v \), for each \( v \in (\bR^{2m})^* \), are triholomorphic
isometries of the hyperKähler structure.

\begin{remark}
  There is a natural circle action on the fibres given by \( x \mapsto
  xe^{\ii t} \).  This rotates the pair \( \omega_J \) and \( \omega_K
  \).  However, the infinitesimal action on \( \alpha \) is \( \alpha
  \mapsto dx\ii - x\ii\oN = dx\ii - x\ii(\oC+\nA) = \alpha\ii +
  2x\nA\ii \).  It follows that under the infinitesimal action \(
  \omega_I \mapsto - 2x\nA\G \wedge \alpha^T \).  Thus the action
  preserves \( \omega_I \) if and only if \( \oN=\oLC \), so \( M \)
  is a flat Kähler manifold.  Thus in general the hyperKähler metric
  obtained from the rigid c-map is different from the hyperKähler
  metrics on cotangent bundles constructed by
  Feix~\cite{Feix:cotangent} and Kaledin~\cite{Kaledin:cotangent}.
\end{remark}

\section{Conic special Kähler manifolds}
\label{sec:conic-special-kahler}

For the local c-map the central starting object is a
\enquote{projective special Kähler manifold}.  We will adopt the usual
strategy
\cite{Freed:special,Cortes-HX:supergravity,Lledo-MvPV:special} of
defining these in terms of conic special Kähler manifolds.

\begin{definition}
  A special Kähler manifold \( (M,g,I,\omega,\nabla) \) is
  \emph{conic} if it admits a vector field \( X \) such that
  \begin{enumerate}
  \item \( g(X,X) \) is nowhere vanishing, and
  \item \( \NB X = -I = \LC X \).
  \end{enumerate}
  We say that a conic structure is \emph{periodic} or
  \emph{quasi-regular} if \( X \) exponentiates to a circle action,
  \emph{regular} if that circle action is free.  We call \( X \) a
  \emph{conic isometry} of~\( C \).
\end{definition}

Note that if \( (C_1,X_1) \) and \( (C_2,X_2) \) conic special Kähler,
then \( (C_1 \times C_2,X_1+X_2) \) is too.  In this way, by
considering regular examples with different periods one gets examples
that are
\begin{enumerate*}
\item quasi-regular, but not regular, or
\item non-periodic
\end{enumerate*}
depending on whether the periods are rationally related or not.

\begin{lemma}
  \label{lem:conic-connection}
  On a conic special Kähler manifold
  \begin{enumerate}[nosep]
  \item the vector field \( X \) is an isometry preserving \( I \), and
  \item the vector field \( IX \) is a homothety that preserves both \( I
    \) and the special connection~\( \NB \).
  \end{enumerate}
\end{lemma}

\begin{proof}
  The results for \( X \) follow purely from \( \LC X = -I \): we have
  \begin{equation*}
    \begin{split}
      (L_Xg)(U,V) &= Xg(U,V) - g([X,U],V) - g(U,[X,V]) \\
      &= (g(\LC_XU,V) + g(U,\LC_XV)) - g(\LC_XU-\LC_UX,V) \eqbreak -
      g(U,\LC_XV-\LC_VX) \\
      &= g(\LC_UX,V)+g(U,\LC_VX) = - g(IU,V) - g(U,IV) = 0
    \end{split}
  \end{equation*}
  and
  \begin{equation*}
    \begin{split}
      (L_XI)U &= [X,IU] - I[X,U] = (\LC_XI)U - \LC_{IU}X + I\LC_UX \\
      &= 0 + U - U = 0.
    \end{split}
  \end{equation*}
  For \( IX \) we have \( \LC(IX) = I\LC X = \Id \) and simple
  modifications of the above arguments show that \( L_{IX}g = 2g \),
  \( L_{IX}I = 0 \).  Its infinitesimal action on \( \NB \) is given
  by
  \begin{equation}
    \label{eq:L-NB}
    \begin{split}
      \MoveEqLeft[1]
      L_{IX}{\NB_UV}-\NB_{[IX,U]}V - \NB_U[IX,V] \\
      &= \NB_{IX}(\NB_UV) - \NB_{\NB_UV}(IX) -
      \NB_{[IX,U]}V - \NB_U(\NB_{IX}V- \NB_V(IX)) \\
      &= R^\NB_{IX,U}V - \NB_{\NB_UV}(IX) + \NB_U(\NB_V(IX)).
    \end{split}
  \end{equation}
  The first term vanishes since \( \NB \) is flat.  For the other
  terms we need to determine \( \NB(IX) \).  Write \( \NB = \LC + \eta
  \).  Putting \( \eta(A,B,C) = g(\eta_AB,C) \) the special Kähler
  conditions imply \( \eta \) is type \( \{3,0\} \) and totally
  symmetric.  In particular, \( \NB X = \LC X \) implies
  \begin{equation}
    \label{eq:X-eta}
    X\hook\eta = 0\quad\text{and hence}\quad  IX\hook \eta = 0.
  \end{equation}
  This gives \( \NB(IX) = \LC(IX) = \Id \).  The remaining part of
  \eqref{eq:L-NB} is thus equal to \( - \NB_UV + \NB_U(V) = 0 \), show
  that \( IX \) preserves~\( \NB \).
\end{proof}

Note that \( X \) itself does not preserve \( \NB \):
\begin{equation*}
  \begin{split}
    (L_X\NB_UV) - \NB_{[X,U]}V - \NB_U[X,V]
    &= -\NB_{\NB_UV}X + \NB_U(\NB_VX) \\
    &= I\NB_UV - \NB_U(IV) = - (\NB_UI)V,
  \end{split}
\end{equation*}
which is only symmetric, not zero.

\begin{lemma}
  The function \( \mu = \tfrac12 g(X,X) \) is both a moment map for
  the conic isometry \( X \) and a Kähler potential for \( g \).
\end{lemma}

\begin{proof}
  To be a moment map we need \( \mu \) satisfy \( d\mu = X \hook
  \omega \).  We have
  \begin{equation*}
    \begin{split}
      (d\mu)(Y)
      &= \tfrac12 Y(g(X,X)) = g(\LC_YX,X) = - g(IY,X) = g(IX,Y) \\
      &= (X\hook\omega)(Y).
    \end{split}
  \end{equation*}
  It follows that
  \begin{equation*}
    \begin{split}
      (dId\mu)(Y,Z)
      &= (dI(X\hook\omega))(Y,Z) = -Yg(X,Z) + Zg(X,Y) + g(X,[Y,Z]) \\
      &= - g(\LC_YX,Z) + g(\LC_ZX,Y) = 2\omega(Y,Z)
    \end{split}
  \end{equation*}
  so \( \omega = \tfrac12dId\mu \) and \( \mu \) is a Kähler potential.
\end{proof}

\begin{definition}
  A \emph{projective special Kähler} manifold is a Kähler quotient \(
  S = C \symp_c X = \mu^{-1}(c)/X \) of conic special Kähler manifold \(
  C \) by a conic isometry \( X \) at some level \( c\in\bR \),
  together with the data necessary to reconstruct \( C \) up to
  equivalence.
\end{definition}

We will not dwell on the extra data needed on \( S \) to specify \( C
\), as we will not need it at this stage.  However, we do note that,
for \( c\ne0 \) the projection \( \mu^{-1}(c) \to S \) is a (pseudo-)
Riemannian submersion, and that the Kähler form~\( \omega_S \) on \( S
\) pulls-back to \( \mu^{-1}(c) \) as~\( \iota^*\omega \), where \(
\iota \colon \mu^{-1}(c) \to C \) is the inclusion.  The Kähler
structure on \( S \) is of Hodge type, since the connection form \(
\varphi = \iota^*X^\flat / g(X,X) = 2\iota^*X^\flat / c \) has
curvature
\begin{equation}
  \label{eq:curv-C}
  d \varphi = 2\iota^*dX^\flat/c = - 4\iota^*\omega /c.
\end{equation}

Let \( \hX \) be the horizontal lift of a conic isometry \( X \) to
the cotangent bundle \( H = T^*C \).  This is the vector field in \(
TT^*C = \mathcal V \oplus \mathcal H = \ker\pi_*\oplus\ker\alpha\),
see~\eqref{eq:V-H}, defined by \( \alpha(\hX) = 0 \), \( \pi_*(\hX) =
X \).  Equip \( H \) with the hyperKähler geometry of
section~\ref{sec:rigid-c-map}.

\begin{proposition}
  \label{prop:tilde-X-JK}
  The horizontal lift \( \hX \) is an isometry of \( H \) preserving
  \( \omega_I \) and with
  \begin{equation*}
    L_{\hX}\omega_J = \omega_K,\qquad L_{\hX}\omega_K = -\omega_J.
  \end{equation*}
\end{proposition}

\begin{proof}
  Let \( \hX \) also denote any choice of lift of~\( \hX \) to a
  vector field on \( \Un(C) \times (\bR^{2m})^* \) or \( \GL(C) \times
  (\bR^{2m})^* \). The essential point is to compute the quantity \(
  d\chi \), where \( \chi_{(u,v)} = \theta_u(\hX) = u^{-1}(X) \).

  On \( \Un(C) \), we claim that
  \begin{equation}
    \label{eq:d-lift}
    d\chi = - \ii \theta - \oN \chi.
  \end{equation}
  To see this, note that \( \chi\colon \GL(C) \to \bR^{2m} \) is the
  equivariant map representing the section~\( X \) of~\( TC \).  It
  follows that \( \NB X \) is represented by the form \( d\chi +
  \oN\chi \in \Omega^1(\GL(C),\bR^{2m}) \).  But \( \NB X = -I \), so
  \( \NB_AX \) is represented by \( u\mapsto u(\theta_u(-(IA)')) \),
  where \( (IA)' \) is any vector on \( \GL(C) \) projecting to \(
  A\in TC \).  On \( \Un(C) \), we have \( \theta_u(-(IA)') =
  u^{-1}(-IA) = -\mathbf i u^{-1}(A) = -\mathbf i\theta_u(A') \) and
  this gives the claimed formula~\eqref{eq:d-lift}.

  We now find
  \begin{equation*}
    L_{\hX}\theta = d(\hX \hook \theta) + \hX \hook d\theta = d\chi
    -\oN(\hX)\theta + \oN\chi = - \ii \theta - \oN(\hX)\theta.  
  \end{equation*}
  Since \( L_{\hX}x = dx(\hX) \), we get
  \begin{equation*}
    L_{\hX}\omega_J = d(L_{\hX} (x \theta)) =
    d(\alpha(\hX)\theta -x\ii \theta) = - d(x\ii\theta)
    = \omega_K. 
  \end{equation*}
  Similarly \( L_{\hX}\omega_K = - \omega_J \).  On the other hand,
  \begin{equation*}
    L_{\hX}\omega_I = d(\hX\hook\omega_I)
    = -\tfrac12 d(\hX \hook \theta^T\wedge\Gi\theta)
    = \pi^* d(X\hook \omega) = 0.
  \end{equation*}

  As the hyperKähler metric is specified by \( \omega_I \), \(
  \omega_J \) and \( \omega_K \), it follows that \( \hX \) is an
  isometry.
\end{proof}

For reference, flatness of \( \NB \) implies
\begin{equation*}
  L_{\hX}\alpha = \hX \hook d\alpha
  = \hX \hook (-\alpha\wedge\oN - x\ON)
  = \alpha \oN(\hX).
\end{equation*}

\section[Twisting hyperKähler manifolds]{Twisting hyperKähler
manifolds by a rotating circle symmetry}
\label{sec:twist-hK}

Let \( (M,g,I,J,K) \) be a hyperKähler manifold.  Suppose that \( X \)
generates a circle action that is isometric, preserves \( I \) but
rotates \( J \) and \( K \).  More precisely assume that
\begin{equation}
  \label{eq:rot-X}
  L_Xg = 0,\quad L_XI = 0\quad\text{and}\quad L_XJ = K.
\end{equation}
We write \( \omega_I(\cdot,\cdot)=g(I\cdot,\cdot) \), etc., for the
Kähler forms. and put
\begin{equation*}
  \alpha_0 = X^\flat = g(X,\cdot),\quad \alpha_A = (AX)^\flat =
  A\alpha_0 \quad\text{for \( A=I,J,K \)}.
\end{equation*}

The question we wish to address is when can this circle action be used
to twist \( (M,g) \) to a quaternionic Kähler metric.  As the twist
construction~\cite{Swann:twist} does not preserve closed forms, we
expect to have to adjust our original structures before twisting.
Therefore consider the metric
\begin{equation*}
  g_N = fg + h\,(\alpha_0^2 + \alpha_I^2 + \alpha_J^2 + \alpha_K^2)
\end{equation*}
for some unknown functions \( f,h \in C^\infty(M) \).  It will be
convenient to allow this metric to be indefinite.  Twisting the
geometry by an unknown curvature form
\begin{equation*}
  F \in \Omega^2_\bZ(M)
\end{equation*}
via a twisting function \( a \in C^\infty(M) \), requires
\begin{equation}
  \label{eq:da}
  da = - X \hook F.
\end{equation}
Let \( W \) be the twist of \( M \) with respect to \( X \), \( F \)
and \( a \).  Topologically \( W = P/\langle X'\rangle \), where \( P
\to M \) is a principal circle bundle with connection one-form \(
\theta_P \) whose is curvature~\( F \).  If the principal action on \( P
\) is generated by \( Y \) and \( \hat X \) is the horizontal
lift of \( X \) to \( P \), then \( X' = \hat X + aY \).  The
geometry on \( W \) is induced from that on \( M \) by pulling
invariant tensors (metrics, complex structures, etc.) back to the
horizontal distribution \( \hor = \ker\theta_P \) and then pushing them
down to the quotient \( W \); we say that such tensor are \Hrelated
and write \( \Hrel \) for this relation.  For this to work we need~\(
X' \) to be transverse to~\( \hor \), which is equivalent to the
non-vanishing of~\( a \).

Note that the original hyperKähler metric has fundamental four form
\begin{equation*}
  \Omega = \omega_I^2 + \omega_J^2 + \omega_K^2,
\end{equation*}
which is invariant under the action of~\( X \).  We write \(
\omega^N_I \) for the \( 2 \)-form defined by \( (g_N,I) \), etc., and
\( \Omega_N \) for the corresponding four-form.  We will look for
twists~\( W \) where \( \Omega_N \) is \Hrelated to a quaternionic
Kähler four-form~\( \Omega_W \).  Pointwise \( \Omega_N \) and \(
\Omega_W \) agree when pulled back to~\( \hor \), which is isomorphic
to both \( T_xM \) and~\( T_yW \).  In dimensions at least \( 12 \),
when \( g_N \) is non-degenerate, the only condition now required for
\( \Omega_W \) to define a quaternionic Kähler metric is that this
four-form be closed \cite{Swann:symplectiques}.  The result we will
prove is:

\begin{theorem}
  \label{thm:qK12}
  Suppose \( X \) is a non-null vector field satisfying
  \eqref{eq:rot-X} on a hyperKähler manifold of dimension at least \(
  8 \).  Then the only twists of \( g_N \) that are quaternionic Kähler
  are given by the data
  \begin{equation*}
    F = k(dX^\flat + \omega_I),\qquad a = k(\norm X^2 - \mu + c),
  \end{equation*}
  with
  \begin{equation*}
    f = \frac B{\mu-c}\quad\text{and}\quad h = -\frac B{(\mu-c)^2},
  \end{equation*}
  where \( \mu \) is a Kähler moment map for the action of \( X \) on
  \( (M,g,I) \) and \( c,k,B \) are constants.
\end{theorem}

We will start by concentrating on the case when \( \dim M \) is at
least \( 12 \).  However, first let us note that it is a simple
consequence that this twist data is usable in all dimensions.

\begin{corollary}
  \label{cor:all}
  In all dimensions, the data \( F \), \( a \), \( f \) and \( h \) of
  Theorem~\ref{thm:qK12} define a twist that is quaternionic Kähler when \(
  g_N \) is non-degenerate.
\end{corollary}

\begin{proof}
  Consider \( M \times \bH^2 \), with circle action on \( \bH^2 \)
  given by \( q = z + jw \mapsto z + je^{-i\theta}w =
  e^{i\theta/2}qe^{-i\theta/2} \).  Then the twist \( W \) is a
  quaternionic submanifold of the twist of \( M\times \bH^2 \), so
  totally geodesic and hence quaternionic Kähler by
  Gray~\cite{Gray:Sp}.
\end{proof}

The constants \( c,k,B \) in Theorem~\ref{thm:qK12} have the following
significance.  Firstly \( B \) is just an overall scaling of the
metric, so only adjusts the result by a homothety.  The scalar~\( k \)
changes the curvature form, and thus affects the topology of the
twist.  Note for global constructions the curvature form~\( F \) must
have integral periods.  Scaling of \( k \) can help to achieve this.
Different choices of~\( k \) then correspond to coverings of the twist
manifold.  Finally \( c \) is the only constant which affects the
local properties of the quaternionic Kähler metric, but it also
affects the global picture by changing the lift of~\( X \) to the
twist bundle.

It follows that the twist construction above agrees with the
constructions of Haydys \cite[via eqn.~(18)]{Haydys:hKS1},
Hitchin~\cite{Hitchin:hK-qK} and Alekseevsky, Cortés and Mohaupt
\cite[via eqn.~(2.4)]{Alekseevsky-CDM:qK-special} and that those
constructions do not admit further variants of the type above.  It
also follows that this construction is inverted by the quaternionic
flip of Haydys~\cite{Haydys:hKS1}.  In particular, given an
isometry~\( Y \) of a quaternionic Kähler manifold~\( Q \), by
\cite{Swann:MathAnn} one may lift \( Y \) to a tri-holomorphic
isometry~\( Y_U \) if the associated bundle \( \UM(Q) \).  Then the
one-dimensional family of corresponding hyperKähler manifolds with
rotating~\( X \) are provided by the hyperKähler quotients of~\(
\UM(Q) \) by \( Y_U \) at different levels.

\subsection{Quaternionic Kähler twists in high dimensions}
\label{sec:quat-kahl-twists}

This section will be devoted to proving Theorem~\ref{thm:qK12} in most
dimensions.  In particular, for \( \dim M \geqslant 12 \) we will
verify that this twist data always leads to a quaternionic Kähler
manifold and we will prove that for this is the only such data that
suffices.

\medbreak First, to determined \( d\Omega_W \) for a general twist, we
need a little more notation, part of which is contained in the next
result.

\begin{lemma}
  The exterior derivatives of \( \alpha_I \), \( \alpha_J \), \(
  \alpha_K \) and \( \alpha_0 = X^\flat \) are
  \begin{gather}
    \label{eq:d-alpha}
    d\alpha_I = 0, \quad d\alpha_J = \omega_K, \quad d\alpha_K =
    -\omega_J
    \shortintertext{and}
    \label{eq:d0}
    d\alpha_0 = G - \omega_I,
  \end{gather}
  for some \( G \in S^2E = \bigcap_{A=I,J,K} \Lambda^{1,1}_A \).
\end{lemma}

\begin{proof}
  The first three assertions follow from \( L_X\omega_I = 0 \), \(
  L_X\omega_J = \omega_K \) via Cartan's formula \( L_X\omega_A =
  X\hook d\omega_A + d(X\hook \omega_A) = 0 + d\alpha_A \).  For the
  final relation, start by noting that the Killing vector field \( X
  \) preserves both the \( \Sp(n)\Sp(1) \)-structure, where \(
  \Sp(n)\Sp(1) \) is the normaliser of \( \Sp(n) \) in \( \SO(4n) \),
  and the Kähler structure \( (g,I) \), which has structure group \(
  \Un(2n)_I \).  It follows that \( \LC X \in (\sP(n)+\sP(1))\cap
  \un(2n)_I \subset \Lambda^2T^*M \).  Under the action of \(
  \Sp(n)\Sp(1) \), we have \( \Lambda^2T^*M = S^2E + S^2H +
  \Lambda^2_0ES^2H \), where \( E \cong \bC^{2n} \) is the fundamental
  representation of \( \Sp(n) \) and \( H \cong \bC^2 \) is that of \(
  \Sp(1) \).  Now the three-dimensional subspace \( \sP(1) = S^2H
  \subset \Lambda^2T^*M \) is spanned by the Kähler forms \( \omega_I
  \), \( \omega_J \) and \( \omega_K \).  With respect to \( I \) we
  have \( \Lambda^{1,1}_IM = S^2E + \bR\omega_I + \Lambda^2_0E\bR_I
  \), with \( S^2E + \bR\omega_I = \sP(n) + \un(1)_I \).  So we
  conclude that \( \LC X \in S^2E + \bR\omega_I \).  To compute the
  coefficient of \( \omega_I \), we note that this component is the \(
  \{2,0\}_J \) part of \( d\omega_I \).  From \( d\alpha_J = \omega_K
  \), we find
  \begin{equation*}
    \begin{split}
      g(KA,B) &= \omega_K(A,B) = d(JX)^\flat(A,B) 
      = g(\LC_A(JX),B) - g(\LC_B(JX),A) \\
      &= - g(\LC_AX,JB) + g(\LC_BX,JA) 
      = -\tfrac12 dX^\flat(A,JB) + \tfrac12 dX^\flat(B,JA) \\
      &= -\tfrac12 (d\alpha_0(JA,B) + d\alpha_0(A,JB)).
    \end{split}
  \end{equation*}
  This implies that \( d\alpha_0^{\{2,0\}_J} = \tfrac12(1-J)d\alpha_0 =
  - g(KJ\cdot,\cdot) = -\omega_I \), as claimed.
\end{proof}

On the quaternionic span of~\( X \), the forms \( \alpha_i \), \(
i=0,I,J,K \), give a volume element
\begin{equation*}
  \vol_\alpha = \alpha_{0IJK} = \alpha_0 \wedge \alpha_I \wedge
  \alpha_J \wedge \alpha_K
\end{equation*}
and \( 2 \)-forms
\begin{equation}
  \label{eq:omega-alpha}
  \oma I = \alpha_{0I} + \alpha_{JK},\quad
  \oma J = \alpha_{0J} + \alpha_{KI}\quad\text{and}\quad
  \oma K = \alpha_{0K} + \alpha_{IJ}.
\end{equation}
With this notation the Hermitian forms of \( g_N \) are
\begin{equation*}
  \omega^N_I = f \omega_I + h \oma I,\qquad\text{etc.}
\end{equation*}

\begin{proposition}
  \label{prop:Omega-W}
  The four-form \( \Omega_W \) \Hrelated to \( \Omega_N \) satisfies
  \begin{equation*}
    \begin{split}
      d\Omega_W &\Hrel d(f^2) \wedge \Omega + 6d(h^2) \wedge
      \vol_\alpha - 2 fh \alpha_I \wedge \Omega \eqbreak + 2 (d(fh) -
      3h^2 \alpha_I) \wedge \sum_{A=I,J,K} \oma A \wedge \omega_A
      \eqbreak + 2fH \wedge \sum_{A=I,J,K} \alpha_A \wedge \omega_A +
      6 hH \wedge \alpha_{IJK}, 
    \end{split}
  \end{equation*}
  where
  \begin{equation}
    \label{eq:HGF}
    H = hG - \frac1a(f + h\norm X^2)F.
  \end{equation}
\end{proposition}

\begin{proof}
  If \( \gamma \in \Omega^p(M)\) is \( X \)-invariant, then the
  \Hrelated form \( \gamma^{}_W \) satisfies
  \begin{equation}
    \label{eq:dW}
    d\gamma^{}_W \Hrel d_W\gamma \coloneqq d\gamma - \frac1aF \wedge X
    \hook \gamma,
  \end{equation}
  see~\cite{Swann:twist}.  Note that \( d_W \) is a derivation on the
  graded algebra \( (\Omega^*(M),\wedge) \) of all forms, so we have
  \begin{equation}
    \label{eq:dOW}
    d\Omega_W \Hrel d_W\Omega_N = 2(d_W\omega^N_I \wedge \omega^N_I +
    d_W\omega^N_J \wedge \omega^N_J + d_W\omega^N_K \wedge \omega^N_K),
  \end{equation}
  even though \( \omega_J \) and \( \omega_K \) are not invariant.  We
  now compute
  \begin{equation*}
    \begin{split}
      d_W\omega^N_I
      &= d(f\omega_I + h\oma I) - \frac1aF\wedge X\hook(f\omega_I +
      h\oma I)\\
      &= df \wedge \omega_I + dh \wedge \oma I + h(d\alpha_0 \wedge
      \alpha_I + \alpha_J\wedge\omega_J + \alpha_K\wedge\omega_K)
      \eqbreak -\frac1a F \wedge(f\alpha_I + h\norm X^2\alpha_I) \\
      &= (df-h\alpha_I) \wedge \omega_I + dh \wedge \oma I +
      h(\alpha_J \wedge \omega_J + \alpha_K \wedge \omega_K) + H
      \wedge \alpha_I,
    \end{split}
  \end{equation*}
  with \( H \) as in~\eqref{eq:HGF}.  Similar computations lead to
  \begin{equation}
    \label{eq:dWomegaN}
    \begin{split}
      d_W
      \begin{pmatrix}
        \omega^N_I\\ \omega^N_J\\ \omega^N_K
      \end{pmatrix}
      &=
      \begin{pmatrix}
        df - h\alpha_I & h\alpha_J & h\alpha_K \\
        -h\alpha_J & df - h\alpha_I & -h\alpha_0 \\
        -h\alpha_K & h\alpha_0 & df - h\alpha_I
      \end{pmatrix}
      \wedge
      \begin{pmatrix}
        \omega_I \\ \omega_J \\ \omega_K
      \end{pmatrix}
      \eqbreak[4]
      + dh \wedge
      \begin{pmatrix}
        \oma I \\ \oma J \\ \oma K
      \end{pmatrix}
      + H \wedge
      \begin{pmatrix}
        \alpha_I \\ \alpha_J \\ \alpha_K
      \end{pmatrix}
      .
    \end{split}
  \end{equation}
  Combining these formulae with~\eqref{eq:dOW} gives the desired
  result.
\end{proof}

We may now confirm that Theorem~\ref{thm:qK12} does indeed give quaternionic
Kähler twists.

\begin{lemma}
  The data of Theorem~\ref{thm:qK12} gives \( d_W\Omega_N = 0 \).
\end{lemma}

\begin{proof}
  Note that \( -\mu + c = -B/f \) and that \( f' = - f^2/B \).  Also
  we have \( F = kG \), so
  equation~\eqref{eq:HGF} becomes
  \begin{equation*}
    \begin{split}
      H
      &= f'G - \frac{f+f'\norm X^2}{\norm X^2 - \mu + c} G \\
      &= \frac{f'(\norm X^2 - B/f) - f - f'\norm X^2}{\norm X^2 - \mu
      + c}F = 0.
    \end{split}
  \end{equation*}
  Now \( d(h^2)\wedge\vol_\alpha = 2hh'\alpha_I\wedge\vol_\alpha = 0
  \), \( d(f^2)=2ff'\alpha_I =2fh\alpha_I \) and \( d(fh) =
  -d(B^2/(\mu -c)^3) = 3B^2/(\mu-c)^4 \,\alpha_I = 3h^2\alpha_I \), so
  \( d_W\Omega_N \) is indeed zero.
\end{proof}

The proof of the uniqueness part of Theorem~\ref{thm:qK12} now proceeds by
decomposing \( d_W\Omega_N = 0 \) in to type components corresponding
to the splitting \( TM = \bH X + \mathscr B \) with \( \mathscr B =
(\bH X)^\bot \).  We call elements of \(
\Span{\alpha_0,\alpha_I,\alpha_J,\alpha_K} \subset T^*M \)
\enquote{type \( (1,0) \)}, and elements in the orthogonal complement
\( T^*\mathscr B \) \enquote{type \( (0,1) \)}.  This induces a type
decomposition of each \( \Lambda^kT^*M \).  In particular, \(
d_W\Omega_N \in \Omega^5(M) \) splits into five components.  Note that
\( \omega \) is type \( (2,0) + (0,2) \), and that \( \Omega \) is
type \( (4,0) + (2,2) + (0,4) \).

Firstly the type \( (0,5) \) part gives
\begin{equation*}
  d(f^2)^{(0,1)} \wedge \Omega^{(0,4)} = 0.
\end{equation*}
The form \( \Omega^{(0,4)} \) is a quaternionic form on \( \mathscr B
\), and so the Lapage like map \( {\Omega^{(0,4)}\wedge\cdot} \colon
\Lambda^k\mathscr B^* \to \Lambda^{k+4}\mathscr B^* \) is injective
for each \( k \leqslant \tfrac12\dim\mathscr B - 2 \) by
Bonan~\cite{Bonan:qH1}.  We conclude that \( d(f^2)^{(0,1)} = 0 \), so
\( d(f^2) \) is type \( (1,0) \).

Now the type \( (1,4) \) part of \( d_W\Omega_N = 0 \) is 
\begin{equation}
  \label{eq:type-1-4}
  (d(f^2) - 2fh\alpha_I) \wedge \Omega^{(0,4)} + 2fH^{(0,2)} \wedge
  \sum_{A=I,J,K} \alpha_A \wedge \omega_A^{(0,2)} = 0.
\end{equation}
This gives directly that \( d(f^2) = \sum_{A=I,J,K} f_A\alpha_A \)
with no \( \alpha_0 \)-component.  Now considering the \( (0,2)
\)-component of 
\begin{equation*}
  0 = d^2(f^2) = \sum_{A=I,J,K} df_A\wedge\alpha_A + f_J\omega_K -
  f_K\omega_J 
\end{equation*}
we see that \( f_J = 0 = f_K \) and so \( d(f^2) = f_I\alpha_I \) with
\( df_I\wedge\alpha_I = 0 \).  Since \( \mu \) is a moment map for the
action of \( X \) with respect to \( \omega_I \), we have \( \alpha_I
= X \hook \omega_I = d\mu \).  We conclude that \( f = f(\mu) \) and
so \( d(f^2) = 2f f'\alpha_I \), where \( {}' \) denotes the
derivative with respect to \( \mu \).

Considering the coefficient of \( \alpha_J \) in \eqref{eq:type-1-4}
we have \( f H^{(0,2)} \wedge \omega_J^{(0,2)} = 0 \), implying that
\( H^{(0,2)} = 0 \).  Now the \( \alpha_I \)-component of
\eqref{eq:type-1-4} reads \( 2f(f'-h) \wedge \Omega^{(0,4)} = 0 \),
giving \( h = f' \).  In particular \( d(h^2)\wedge\vol_\alpha =
2f''f'\alpha_I\wedge\vol_\alpha = 0 \).

Using this information, we have
\begin{equation}
  \label{eq:dWO2}
  \begin{split}
    0 = d_W\Omega_N &= \bigl((f^2)'' - 6{f'}^2\bigr)\alpha_I \wedge
    \sum_{A=I,J,K} \oma A \wedge \omega_A \eqbreak + 2f H\wedge \sum_{A=I,J,K}
    \alpha_A \wedge \omega_A + 6f' H\wedge \alpha_{IJK},
  \end{split}
\end{equation}
with \( H^{(0,2)} = 0 \).  Taking the \( (2,3) \)-part of this
equation gives
\begin{equation}
  \label{eq:type-2-3}
  2fH^{(1,1)} \wedge (\alpha_I \wedge \omega_I^{(0,2)} + \alpha_J
  \wedge \omega_J^{(0,2)} + \alpha_K \wedge \omega_K^{(0,2)})
\end{equation}
Writing \( H^{(1,1)} = \sum_{i=0,I,J,K} \alpha_i \wedge H_i \), with
\( H_i \) of type \( (0,1) \), equation~\eqref{eq:type-2-3} becomes
\begin{equation*}
  H_0 \wedge \omega_A^{(0,2)} = 0\quad\text{and}\quad
  H_A \wedge \omega_B^{(0,2)} = H_B \wedge
  \omega_A^{(0,2)},
\end{equation*}
for \( A,B=I,J,K \).  As \( \dim\mathscr B \geqslant 8 \), we
conclude that \( H^{(1,1)} = 0 \).

We now have that \( H \) is type \( (2,0) \), so \(
H\wedge\alpha_{IJK} = 0 \) and the remaining terms
in \( d_W\Omega_N \) are all of type \( (3,2) \).  Write \( H =
\sum_{0\leqslant i<j\leqslant K} H_{ij}  \alpha_i \wedge \alpha_j \).
The coefficient of \( \omega_I^{(0,2)} \) in~\eqref{eq:dWO2} is
\begin{equation*}
  ((f^2)'' - 6{f'}^2)\alpha_{IJK} + 2f(H_{JK}\alpha_{IJK} +
  H_{0J}\alpha_{0JI} + H_{0K}\alpha_{0KI}) = 0,
\end{equation*}
so \( H_{0J} = 0 = H_{0K} \) and \( H_{JK} = (6{f'}^2-(f^2)'')/2f \).
On the other hand, the coefficient of \( \omega_J^{(0,2)} \) is
\begin{equation*}
  -((f^2)'' - 6{f'}^2)\alpha_{0IJ} + 2f(H_{IK}\alpha_{IKJ} +
  H_{0I}\alpha_{0IJ}) = 0, 
\end{equation*}
giving \( H_{IK} = 0 \) and \( H_{0I} = - H_{JK} \).  Finally, the
coefficient of \( \omega_K^{(0,2)} \) leads to \( H_{IJ} = 0 \).  Thus
\begin{equation}
  \label{eq:Hf}
  H = \frac1{2f}((f^2)'' - 6{f'}^2)\omam I,
\end{equation}
where \( \omam I \coloneqq \alpha_{0I} - \alpha_{JK}\).  Using the
definition~\eqref{eq:HGF} of \( H \) shows that
\begin{equation}
  \label{eq:Ff}
  F = \frac a{f+f'\norm X^2}\Bigl(f'G -
  \frac1{2f}((f^2)''-6{f'}^2)\omam I\Bigr). 
\end{equation}
Note that non-degeneracy of \( g_N \) ensures that \( (f+f'\norm X^2)
= (f+h\norm X^2) \) and \( f \) are non-zero.

The two-form \( F \) in~\eqref{eq:Ff} needs to be closed and to
satisfy~\eqref{eq:da}.  To evaluate these conditions,
introduce the function \( R = f'/f \).  Then \( R' = f''/f -
{f'}^2/f^2 \), \( dR = R'\alpha_I \) and
\begin{equation}
  \label{eq:FR}
  F = \frac a{1+R\norm X^2}(RG + (R^2-R')\omam I).
\end{equation}
Note that \( G = dX^\flat + \omega_I \) and that \( d\norm X^2 = -
X\hook dX^\flat \).  Equation~\eqref{eq:FR} then implies that
\begin{equation*}
  \begin{split}
    X \hook F
    &= \frac a{1+R\norm X^2}(R(-d\norm X^2+\alpha_I) + \norm X^2
    (R^2-R')\alpha_I) \\ 
    &= \frac a{1+R\norm X^2}(-d(1+R\norm X^2) + R(1+R\norm
    X^2)\alpha_I) \\
    &= a\, (-d\log(1+R\norm X^2) + R\alpha_I)
    = a\, (d\log f - d\log(1+R\norm X^2)) \\
    &= - a\,\, d\log((1+R\norm X^2)/f),
  \end{split}
\end{equation*}
Thus equation~\eqref{eq:da} gives \( d \log a = d \log((1+R\norm
X^2)/f) \), since the twist construction requires \( a \) to be
non-zero.  We conclude that
\begin{equation*}
  a = \frac {k_0}f(1+R\norm X^2)
\end{equation*}
for some non-zero constant~\( k_0 \).

It remains to determine when \( F \) is closed.  When \( k_0 \ne 0 \),
we substitute the expression for \( a \) into \eqref{eq:Ff}.  This
gives
\begin{equation*}
  \begin{split}
    0
    &= \frac1{k_0} dF
    =  d\Bigl(\frac 1f(RG + (R^2-R')\omam
    I)\Bigr) \\
    &= -\frac Rf\alpha_I\wedge (RG+(R^2-R')\omam I)
    + \frac1f(R'\alpha_I\wedge G - (R^2-R')'\alpha_{IJK}) \eqbreak +
    \frac1f(R^2-R')((G-\omega_I)\wedge\alpha_I - \omega_K\wedge\alpha_K
    - \alpha_J\wedge\omega_J) \\
    &= \frac1f (R(R^2-R') - (R^2-R')')\alpha_{IJK} -
    \frac1f(R^2-R')\sum_{A=I,J,K}\alpha_A\wedge\omega_A,
  \end{split}
\end{equation*}
since the coefficient of \( G\wedge\alpha_I \) sums to zero.  Taking
the \( (1,2) \) component, we see that \( R^2-R' = 0 \) and that this
is the only equation that needs to be satisfied.  But this gives
either \( R = 0 \), so \( f' = 0 \), \( F = 0 \) and the twist is
trivial, or \( R = 1/(-\mu + c) \) and \( f = B/(\mu - c) \) for
constants \( c \) and~\( B \).  This latter case then has \( h = f' =
-B/(\mu - c)^2 \) and \( a = k (\norm X^2 - \mu + c) \), with \( k =
-k_0/B \).  Finally we see \( F = (Ra/(1+R\norm X^2))G = kG =
k(dX^\flat + \omega_I) \) and we have the claimed result.

\subsection{Uniqueness in dimension eight}
\label{sec:uniq-dimens-eight}

Here we extend Theorem~\ref{thm:qK12} to manifolds of dimension~\( 8 \).  By
Corollary~\ref{cor:all}, we know that the given twist data does lead to a
quaternionic Kähler metric.  It thus remains to prove that this twist
data is unique.

Recall \cite{Swann:MathAnn} that an almost quaternion Hermitian
structure \( (W,g_W,\Omega_W) \) in dimension \( 8 \) is quaternionic
Kähler if and only if its fundamental four-form is closed and \(
d(S^2H) \subset T^*W\wedge S^2H \subset \Omega^3(W) \).  The latter
condition is the requirement that locally the Hermitian forms \(
\omega^W_I,\omega^W_J,\omega^W_K \) associated to the compatible local
almost complex structures generate a differential ideal.  With respect
to such a local triple of Hermitian forms the quaternionic Kähler
condition in all dimensions at least \( 8 \) is equivalent to the
existence of a local one-form \( \sigma = (\sigma_{AB}) \in
\Omega^1(U,\so(3)) \) such that \( \omega^W = ( \omega^W_I,
\omega^W_J, \omega^W_K)^T \)
\begin{equation}
  \label{eq:qK}
  d\omega^W = \sigma \wedge \omega^W.
\end{equation}

Our problem with computing the derivatives of \( \omega^W_A \) is that
they do not correspond to invariant forms on~\( M \), and so we can
not directly use the formulae of \cite{Swann:twist}.  To resolve this
first note that it is sufficient to work away from the fixed point set
of \( X \): this is a totally geodesic submanifold of codimension at
least two, so its complement is open and dense and the complement
corresponds to an open dense subset of the twist.  

Let \( \theta(t,q) = \theta_t(q) \), for \( t\in\bR \) and \( q\in M
\), be the one-parameter group generated by~\( X \).  On \( M \), we
have \( \theta_t^*\omega_J = \cos(t)\omega_J + \sin(t)\omega_K \) and
\( \theta_t^*\omega_K = -\sin(t)\omega_J + \cos(t)\omega_K \).

In a neighbourhood of a point~\( p \) outside of the fixed-point set,
we may choose a slice \( \mathscr S_p \) transverse to the orbits
of~\( X \) and an open neighbourhood \( U_p \) of \( \{0\} \times
\mathscr S_p \) in \( \bR \times \mathscr S_p \) such that \( \theta
\colon U_p \to M \) is an embedding with image an open set \( B_p \).
On \( B_p \), we define invariant forms \( \omega_\tJ \), \(
\omega_\tK \) as the translates under \( \theta_t \) of \( \omega_J \)
and \( \omega_K \) over \( \mathscr S_p \).  We have
\begin{equation*}
  \omega_\tJ = u\omega_J + v\omega_K,\quad \omega_\tK
  = -v\omega_J + u\omega_K 
\end{equation*}
for some functions \( u,v \in C^\infty(B_p) \) satisfying \( u^2+v^2=1
\).  Putting
\begin{equation*}
  \gamma =
  \begin{pmatrix}
    1&0&0\\
    0&u&v\\
    0&-v&u
  \end{pmatrix}
\end{equation*}
and \( \tilde\omega = (\omega_I,\omega_\tJ,\omega_\tK)^T \), we write
\( \tilde\omega^N = \gamma \omega^N \) and define \( \tilde\omega^W \)
by \( \tilde\omega^W \Hrel \tilde\omega^N \).  Now \(
d_W\tilde\omega^N = d\gamma\wedge \omega^N + \gamma\, d_W\omega^N \)
and it follows that \( \tilde \omega^W \) satisfies the quaternionic
Kähler condition~\eqref{eq:qK} \( d\tilde\omega^W = \tilde\sigma
\wedge \tilde\omega^W \) if and only if
\begin{equation}
  \label{eq:dWomegaNq}
  d_W\omega^N = \sigma^N\wedge\omega^N
\end{equation}
for a \( \sigma^N \in \Omega^1(B_p,\so(3)) \).  The \( \so(3)
\)-connections \( \sigma^N \) and \( \tilde\sigma^W \) satisfy the
gauge type relation \( \tilde\sigma^W \Hrel \gamma\sigma^N\gamma^{-1} -
(d\gamma)\gamma^{-1} \).

Now to solve~\eqref{eq:dWomegaNq}, we use~\eqref{eq:dWomegaN}.  As in
the proof of Theorem~\ref{thm:qK12}, we consider the components of
\eqref{eq:qK} according to their types with respect to the splitting
of \( \Lambda^kT^*M \) induced by \( \Span{\alpha_0,\dots,\alpha_K}
\subset T^*M \) and its orthogonal complement.

Write \( V = \norm X^{-2} \), so that \( {\omega^N} = f\omega^\beta +
(Vf+h)\omega^\alpha \).  Then, the \( (2,1) \)-component
of~\eqref{eq:dWomegaN} gives
\begin{equation*}
  (Vd^{(0,1)}f + d^{(0,1)}h)\wedge \oma A + H^{(1,1)}\wedge \alpha_A =
  \sum_{B=I,J,K}(\sigma^N_{AB})^{(0,1)} \wedge (Vf+h) \oma B
\end{equation*}
for each \( A=I,J,K \).  As each non-zero element in the span of \(
\oma I \), \( \oma J \) and \( \oma K \) is non-degenerate, it follows, that 
\begin{equation}
  \label{eq:8H11}
  H^{(1,1)} = 0.
\end{equation}
Now the fact that \( \sigma^N \) is skew-symmetric yields
\begin{equation}
  \label{eq:801}
  (\sigma^N)^{(0,1)} = 0 \quad\text{and}\quad
  Vd^{(0,1)}f+d^{(0,1)}h = 0. 
\end{equation}

Now consider the \( (1,2) \)-component of~\eqref{eq:dWomegaNq}.  We
have immediately that \( H^{(0,2)} = H^{02}_I\wedge\omb I +
H^{02}_J\wedge\omb J + H^{02}_K\wedge\omb K \) for some \( H^{02}_A
\).  Considering the coefficients of~\( \omb A \) and putting \(
\delta = d^{(1,0)}f-h\alpha_I \), the equations then give
\begin{gather*}
  \begin{pmatrix}
    \delta +  H^{02}_I\alpha_I & h\alpha_J + H^{02}_J\alpha_I
    & h\alpha_K + H^{02}_K\alpha_I\\
    -h\alpha_J + H^{02}_I\alpha_J & \delta + H^{02}_J\alpha_J
    & -h\alpha_0 + H^{02}_K\alpha_J\\
    -h\alpha_K + H^{02}_I\alpha_K & h\alpha_0 + H^{02}_J\alpha_K
    & \delta + H^{02}_K\alpha_K
  \end{pmatrix}
  =f\sigma^N.
\end{gather*}
The off-diagonal terms and the skew-symmetry of \( \sigma^N \) imply
that \( H^{(0,2)} = 0 \).  It then follows that \( \delta = 0 \),
i.e.,
\begin{equation}
  \label{eq:8d10f}
  d^{(1,0)}f = h\alpha_I,
\end{equation}
and that
\begin{equation}
  \label{eq:8sN}
  f\sigma^N_{IJ} = h\alpha_J,\quad f\sigma^N_{IK} = h\alpha_K
  \quad\text{and}\quad f\sigma^N_{JK} = -h\alpha_0.
\end{equation}
Using \( H^{(0,2)} = 0 \) and \( (\sigma^N)^{(0,1)}=0 \), the \( (0,3)
\)-component of~\eqref{eq:dWomegaNq} implies \( d^{(0,1)}f = 0 \).
Together with \eqref{eq:8d10f}, we thus have \( df = h\alpha_I \), so
\( f = f(\mu) \) and \( h = f' \).

Finally, all that remains of~\eqref{eq:dWomegaNq} is the \( (3,0)
\)-component.  This reduces to
\begin{equation*}
  h'\alpha_I\wedge\omega^\alpha + H\wedge\alpha = h\sigma^N\wedge\omega^\alpha.
\end{equation*}
Multiplying through by~\( f \), we use~\eqref{eq:8sN} to get
\begin{gather*}
  fH\wedge\alpha_I = (2h^2-fh')\alpha_{IJK},\quad
  fH\wedge\alpha_J = -(2h^2-fh')\alpha_{0IJ},\\
  fH\wedge\alpha_K = -(2h^2-fh')\alpha_{0IK}.
\end{gather*}
The first of these equations implies that \( fH =
(2h^2-fh')\alpha_{JK} +
\alpha_I\wedge(\lambda_0\alpha_0+\lambda_J\alpha_J+\lambda_K\alpha_K)
\), the second and third equations then give \( \lambda_0 = 2h^2-fh'
\), \( \lambda_J = 0 = \lambda_K \).  We conclude that \( fH =
-(2h^2-fh')(\alpha_{0I}-\alpha_{JK}) \).  Substituting \( h = f' \),
we see that \( H \) is given by equation~\eqref{eq:Hf} as in the
higher-dimensional case.  The arguments following~\eqref{eq:Hf}, then
provide the claimed uniqueness in dimension~\( 8 \).

\section{Geometry of the twist}
\label{sec:geometry-twist}

Let us start by specialising the above results when the hyperKähler
manifold is the image of the rigid c-map for \( C \) conic special
Kähler.  By Alekseevsky et al.\ \cite{Alekseevsky-CDM:qK-special} and
the remarks after Corollary~\ref{cor:all}, we now know this gives the
c-map and its one loop deformations via~Figure~\ref{fig:spaces}.  We
wish to show how the properties listed by Ferrara and Sabharwal
\cite{Ferrara-S:q} may be obtained from the twist picture and describe
some global aspects.

Let \( X \) be the conic isometry, and put \( H = T^*C \).  By
Proposition~\ref{prop:tilde-X-JK}, we have that the horizontal lift \( \hX \) is
a hyperKähler isometry of \( H \) rotating \( J \) and~\( K \).

\begin{lemma}
  \label{lem:cone-twist}
  The twist data for the symmetry \( \hX \) of \( H = T^*C \) is
  \begin{equation*}
    F = \tfrac 12k(\alpha \wedge \Gi \alpha^T + \theta^T \wedge \Gi
    \theta), \qquad
    a = k(\mu + c), \qquad \mu = \tfrac12 \norm{\hX}^2.
  \end{equation*}
  Moreover, \( F \)~is exact.
\end{lemma}

\begin{proof}
  We first compute \( d\hX^\flat \).  In the notation of
  Proposition~\ref{prop:tilde-X-JK}, we have \( \alpha_0 = \hX^\flat = \hX \hook
  \theta^T\G\theta = \chi^T\G\theta \).  Thus
  \begin{equation}
    \label{eq:dX-flat}
    \begin{split}
      d\alpha_0
      &= (d\chi)^T\wedge\G \theta + \chi^T\G d\theta 
      = (-\ii\theta - \oN\chi)^T \wedge \G\theta -
      \chi^T\G\oN\wedge\theta \\
      &= \theta^T \wedge\ii\G \theta - \chi^T(\oN^T\G + \G\oN)\wedge\theta \\
      &= \theta^T \wedge\Gi \theta
    \end{split}
  \end{equation}
  by Proposition~\ref{prop:tilde-X-JK} and~\eqref{eq:oN-G-th}.  As \( F =
  k(d\hX^\flat + \omega_I) \), the claimed expression for \( F \)
  follows from~\eqref{eq:oI}: \( \omega_I =
  \tfrac12(\alpha\wedge\Gi\alpha^T - \theta^T\wedge\Gi\theta) \).  To
  check the formula for~\( \mu \), we wish to show that \( d\mu =
  \hX \hook \omega_I \).   We compute
  \begin{equation*}
    \begin{split}
      d \norm{\hX}^2
      &= d\{(\theta^T\G\theta)(\hX,\hX)\}
      = d\{\chi^T\G\chi\} \\
      &= (-\ii\theta-\oN\chi)^T\G\chi + \chi^T\G(-\ii\theta-\oN\chi)\\
      &= \hX \hook (- \theta^T\wedge\Gi \theta) + \hX \hook
      \chi^T\{(\oN^T\G + \G\oN)\wedge\theta\} \\
      &= 2\hX \hook \omega_I,
    \end{split}
  \end{equation*}
  where we again have used~\eqref{eq:oN-G-th}.  It follows that \( \mu
  = \tfrac12\norm{\hX}^2 \), as claimed, and the expression for \( a
  \) is now obtained from \( a = k(\norm{\hX}^2 - \mu + c) \).

  To show that \( F \) is exact, it is enough to show \(
  \alpha\wedge\Gi\alpha^T \) is exact, since \(
  \theta^T\wedge\Gi\theta = d\alpha_0 \).  But \( d(\alpha\Gi x^T) =
  -\alpha\wedge\oN\Gi x^T - \alpha\Gi dx^T = \alpha \wedge \Gi
  (\oN^Tx^T-dx^T) = \alpha \wedge \Gi\alpha^T \), by~\eqref{eq:oN-G},
  giving \( F = d(\tfrac12k(\alpha\Gi x^T + \alpha_0)) \) as desired.
\end{proof}

We now wish to describe the geometric properties of the quaternionic
Kähler manifolds~\( Q_c \) that we obtain from the above twist
construction.  We put \( Q = Q_0 \).

\begin{lemma}
  For general \( c \), the metric \( g_N \) is
  \begin{equation*}
    g_N = \frac B{\mu+c}\Bigl(g_\bot - \frac {\mu-c}{\mu+c}\,g_{\bH X}\Bigr)
  \end{equation*}
  where \( B \)~is an arbitrary constant, \( g_{\bH X} \) is the
  restriction of \( g \) to \( \bH X \subset TH \), and \( g_\bot \)
  is the restriction to \( (\bH X)^\bot \).  

  In particular, for \( c=0 \), the metric \( g_N \) is
  \begin{equation*}
    g_N = \frac B\mu(g_\bot - g_{\bH X}).
  \end{equation*}
\end{lemma}

\begin{proof}
  By Theorem~\ref{thm:qK12}, \( g_N = fg + hg_\alpha \), where \( g_\alpha =
  \alpha_0^2 + \alpha_I^2 + \alpha_J^2 + \alpha_K^2 \).  As \( g =
  g_\bot + g_{\bH X} = g_\bot + \tfrac1{2\mu} g_\alpha\), the
  relations \( f = B/(\mu+c) \) and \( h = -B/(\mu+c)^2 \) imply the
  claimed result.
\end{proof}

Since \( g_{\bH X}/\mu \) is positive definite, it is natural to take
\begin{equation*}
  B = -1
\end{equation*}
when \( c=0 \).  In all cases, we may take \( k=1 \), since \( F \)
and \( a \) only occur as the combination \( \tfrac1a F \), and so the
\( k \)'s cancel.  In fact, this choice is good topologically.

\begin{proposition}
  \label{prop:topology}
  Let \( k=1 \) in~Lemma~\ref{lem:cone-twist}.  For \( c=0 \), the twist of
  \( H = T^*C \) is diffeomorphic to the product \( (H/\Span{\hX})
  \times S^1 \) as diffeological spaces.  
\end{proposition}

In particular, when \( H/\Span{\hX} \) is smooth, the diffeomorphism
is as manifolds.  For \( X \) regular, this is the case when \( C \)
has a global flat symplectic frame; such a frame always exists on some
discrete cover of~\( C \).

\begin{proof}
  From Lemma~\ref{lem:cone-twist}, \( F = d\beta \) with \( \beta =
  \tfrac12(\alpha\Gi x^T + \alpha_0) \).  This has the property that
  \( \hX \hook \beta = \tfrac12\alpha_0(\hX) = \tfrac12\norm{\hX}^2 =
  \mu \) which is the twisting function~\( a \).  As \( F \) is exact,
  the twist bundle~\( P \) is trivial, \( P = H \times S^1 \), with
  connection one-form \( \phi = \beta + d\tau \), where \( \tau \) is
  the parameter on the \( S^1 \)-factor.  The twist~\( W \) is the
  quotient \( P/\Span{X'} \), where \( X' = \hX^\phi +
  a\tfrac\partial{\partial\tau} \) with \( \hX^\phi \) the \( \phi
  \)-horizontal lift of \( \hX \) to~\( P \).  This means that in the
  product structure \( P = H \times S^1 \), we have \( \hX^\phi = \hX
  + \lambda \tfrac\partial{\partial\tau} \) and \( 0 = \phi(\hX^\phi)
  = \beta(\hX) + \lambda = a + \lambda \).  Thus we have \( X' = \hX
  \) and \( W = (H/\Span{\hX}) \times S^1 \).
\end{proof}

Note that for general~\( c \), we get \( X' = \hX -
c\tfrac\partial{\partial\tau} \) above and the twist is \( (H\times
S^1)/\Span{\hX - c\tfrac\partial{\partial\tau}} \).

Now in the general twist construction, if \( N \subset M \) is an \( X
\)-invariant submanifold, it inherits natural twist data from~\( M \):
writing \( \iota \colon N \to M \) for the inclusion, the circle
bundle is \( P_N = \iota^*P \) with curvature \( \iota^*F \) preserved
by~\( X \), and the twist function is simply \( \iota^*a \), as \(
d\iota^*a = \iota^*da= -\iota^*X\hook F = - X \hook \iota^*F \).  This
implies that the lift \( X' \) of \( X \) to~\( P \) its tangent to
the submanifold~\( P_N \).  In particular, if \( X' \) is
regular, then the twist \( P_N/\Span{X'} \) of~\( N \) is a
submanifold of the twist \( P/\Span{X'} \) of~\( M \).

We first consider the image of the cone \( C \) under the twist.

\begin{proposition}
  If \( X \) is regular, the image~\( C_Q \) in the twist \( Q \) of
  the conic manifold~\( C \) at \( c=0 \) is a Kähler product of the
  projective special Kähler manifold~\( S \) and the quotient \(
  \CH(1)/\bZ \) of a one-dimensional complex hyperbolic space.
\end{proposition}

\begin{proof}
  Let \( \iota_C\colon C \to H \) be the inclusion, where \( C \) lies
  in \( H = T^*C \) as the zero-section.  The twist data is \(
  \iota^*F = -\iota^*\omega \), \( \iota^*a = \iota^*\mu = s \) and
  \begin{equation*}
    \iota^*g_N = -\tfrac1\mu g_{C,\bot} + \tfrac1{2\mu^2}(\alpha_0^2 +
    \alpha_I^2).
  \end{equation*}
  Let \( S \) be the Kähler quotient \( \mu^{-1}(s)/X \) of \( C \) at
  some regular value~\( s \).  As \( C \) is conic with \( X \)
  regular, we have projections \( \pi_S\colon C \to S \) and \(
  \overline \pi_S\colon H \to S \); furthermore, \( \pi_S^*g_S^{} =
  g_{C,\bot} \) on \( \mu^{-1}(s) \).  Since \( L_{IX}g = 2g \), we
  have \( L_{IX}(g/\mu) = 0 \) and so \( \pi_S^*g_S^{} = s
  g_{C,\bot}/\mu \) on all of~\( C \).

  Suppose \( \gamma \in \Omega^*(S) \).  Then \(
  \overline\pi_S^*\gamma \) is \( X \)-invariant, and there is an
  \Hrelated differential form \( \gamma_Q \) on~\( Q \).  We have \(
  d\gamma_Q \Hrel \overline\pi_S^*d\gamma - \tfrac1aF \wedge X\hook
  \pi_S^*\gamma = \overline\pi_S^*d\gamma \).  Thus the set
  \begin{equation}
    \label{eq:Om-S}
    \Omega^*_S = \Set{\gamma_Q \in \Omega^*(Q)}{\exists
    \gamma\in\Omega^*(S) \text{ such that } \gamma_Q \Hrel
    \overline\pi_S^*\gamma}
  \end{equation}
  is a differential subalgebra and pulls back to a subalgebra of \(
  \Omega^*(C_Q) \).

  Putting \( \tilde\alpha_0 \Hrel \alpha_0/\mu \) and \(
  \tilde\alpha_I \Hrel \alpha_I/\mu \), we use \eqref{eq:dX-flat} to
  compute
  \begin{equation}\label{eq:d-t-alpha}
    \begin{split}
      d\tilde\alpha_0 \Hrel {} &-\frac1{\mu^2} d\mu \wedge \alpha_0 +
      \frac1\mu d\alpha_0 - \frac1{\mu^2}F\wedge\alpha_0(X) \\
      &= -\frac1{\mu^2}\alpha_I\wedge\alpha_0 + \frac1\mu d\alpha_0 -
      \frac2\mu (d\alpha_0 + \omega_I) \\
      &= \frac1{\mu^2}\alpha_0\wedge\alpha_I -
      \frac1\mu(\alpha\wedge\Gi\alpha^T)
    \end{split}
  \end{equation}
  and
  \begin{equation*}
    d\tilde\alpha_I \Hrel d(1/\mu)\wedge\alpha_I = 0.
  \end{equation*}

  Pulling back to \( C_Q \), we have that \( d\iota^*\tilde\alpha_0 =
  \iota^*\tilde\alpha_0\wedge\iota^*\tilde\alpha_I \) and so \(
  \mathcal S_1 = \ker(\iota^*\alpha_0)\cap\ker(\iota^*\alpha_I) \) is
  an integrable foliation of \( C_Q \).  Similarly the kernel of \(
  \iota^*(\Omega^1_S) \) is a complementary two-dimensional integrable
  distribution~\( \mathcal S_2 \) of~\( C_Q \).  The leaves of \(
  \mathcal S_1 \) have the same ring of functions as~\( S \), so are
  diffeomorphic to~\( S \).  Moreover the two distributions \(
  \mathcal S_1 \) and \( \mathcal S_2 \) are orthogonal in the induced
  metric.  The metric on \( \mathcal S_1 \) is \Hrelated to \(
  -g_{C,\bot}/\mu = g_S/s \), that on \( \mathcal S_2 \) is \(
  \tfrac12\iota^*(\tilde\alpha_0^2+\tilde\alpha_I^2) \).  Each
  distribution inherits a complex structure from~\( I \) and we obtain
  a Kähler product.  The differential relations, together with the
  completeness of the metric on \( S_2 \), show that the universal
  cover of \( \mathcal S_2 \) is \( \CH(1) \) as a solvable group.  By
  Proposition~\ref{prop:topology}, \( \mathcal S_2 = \bR_{>0} \times S^1 =
  \CH(1)/\bZ \).
\end{proof}

\begin{proposition}
  Let \( X \)~be regular and \( c=0 \).  The differential algebra \(
  \Omega^*_S \) in~\eqref{eq:Om-S} gives a distribution
  \begin{equation*}
    \mathcal D = \ker\Omega^1_S
  \end{equation*}
  on \( Q \) that is integrable.  Its leaves are discrete quotients of
  complex Heisenberg groups~\( \CH(m+1) \), \( \dim_\bC C = m \), with
  left-invariant Riemannian structures homothetic to the standard
  Kähler metric.
\end{proposition}

\begin{proof}
  Integrability of \( \mathcal D \) follows directly from that fact
  that \( \Omega^*_S \) is closed under the exterior differential.
  Each leaf~\( L_Q \) of~\( D \) is a twist of a leaf~\( L \) of \(
  \overline\pi_S^*\Omega^1(S) \) on \( H = T^*C \).  Now \( C_L
  \coloneqq L \cap C \) is an orbit of \( X \) and \( IX \), so by
  regularity is a copy of \( \bC^* \).

  Write \( \iota\colon C_L \to C \) for the inclusion of this orbit.
  The leaf \( L \) is the restriction \( \iota^*T^*C = T^*C|_{C_L} \).
  The relations in~\eqref{eq:X-eta}, show that \( \iota^*\NB \) and \(
  \iota^*\LC \) agree as connections on the bundle~\( \iota^*TC \).
  In particular, this bundle is flat, and there is a discrete cover \(
  C_L' \) over which it is trivial.  We thus have that the pull-back
  \( \iota^*\Sp(C) \) of the bundle of symplectic frames admits a flat
  symplectic section~\( s \) over~\( C_L' \).  In our case \( s \)~may
  be chosen to be unitary with span of the negative definite
  directions equal to the span of \( X \) and~\( IX \).  As in
  \eqref{eq:local} each \( v \in (\bR^{2m})^* \) gives rise to a
  tri-holomorphic isometry of \( H' \) around \( C_L' \).  Write \(
  U_v \) for the corresponding vector field.

  Write \( \mu = \varepsilon \nu^2 \), for some \( \varepsilon \in
  \{\pm1\} \) and a smooth function \( \nu\colon L'\to\bR \).  Put \(
  V_v = \nu U_v \). We claim that \( X \), \( IX \), \( V_v \), \(
  v\in (\bR^{2m})^* \), are \Hrelated to vector fields \( X_Q \), \(
  IX_Q \), \( (V_v)_Q \) on \( L_Q' \) generating a complex Heisenberg
  algebra.

  If \( A \) and \( B \) are vector fields on~\( L' \) that are
  \Hrelated to \( A_Q \) and \( B_Q \) on \( L_Q' \), then the Lie
  brackets are related by
  \begin{equation*}
    [A_Q,B_Q] \Hrel [A,B] +
    \frac 1{\iota^*a}(\iota^*F)(A,B)X,
  \end{equation*}
  see \cite[Lemma~3.7]{Swann:twist}.  In our case, from
  Lemma~\ref{lem:cone-twist} we have \( a = \mu = \varepsilon\nu^2 \) and
  \( F = \tfrac12(\alpha\wedge\Gi\alpha^T + \theta^T\wedge\Gi\theta) =
  \omega^* - \omega \).

  Firstly, \( \iota^*F(IX,X) = -i^*\omega(IX,X)= g(X,X) = 2\mu
  \). Hence,
  \begin{equation*}
    [(IX)_Q,X_Q] = \frac1\mu \iota^*F(IX,X)X_Q = 2X_Q.
  \end{equation*}

  Considering the Lie brackets of vertical vector fields, we find
  \begin{equation*}
    [(V_v)_Q,(V_w)_Q] = \frac
    1\mu\nu^2\omega^*(U_v,U_w)X_Q = \frac\varepsilon2(v\Gi w^T)X_Q.
  \end{equation*}
  For mixed brackets, note that \( X \) and \( IX \) commute with the
  fibre translations~\( U_v \).  Now \( d\nu \) is given by \(
  X\hook\omega_I = d\mu = 2\varepsilon\, \nu d\nu \), so \( X\nu = 0
  \) and \( IX\nu = \omega_I(X,IX)/2\varepsilon\nu = \nu \).  We thus
  have
  \begin{equation*}
    [X_Q,(V_v)_Q] = 0 
  \end{equation*}
  and
  \begin{equation*}
    [(IX)_Q,(V_v)_Q] = [IX,(V_v)]_Q = (IX\nu)\,(U_v)_Q
    = \nu (U_v)_Q = (V_v)_Q.
  \end{equation*}
  We thus see that \( (IX)_Q \), \( X_Q \), \( (V_v)_Q \), \(
  v\in(\bR^{2m})^* \), span a Lie algebra~\( \lie g \); \( (IX)_Q \)
  acts with weight \( 2 \) on~\( X_Q \) and weight \( 1 \) on~\(
  (V_v)_Q \); furthermore \( X_Q \), \( (V_v)_Q \) span a Heisenberg
  algebra with derived algebra generated by \( X_Q \).  Thus \( \lie g
  \) is the solvable algebra of \( \CH(m+1) \) as shown by Hitchin in
  \cite{Hitchin:quaternionic}.  This is a discrete cover of the
  leaf~\( L_Q \).

  For the metric \( g_N \) and our choice of section~\( s \), the
  above Lie algebra generators, where \( v \) runs over the standard
  unitary basis for \( \G \) on \( (\bR^{2m})^* \), has constant
  coefficients and so is left-invariant.  Indeed this homothetic to
  the standard orthonormal basis of \( \CH(m+1) \) when with a scaling
  factor of~\( \sqrt2 \).
\end{proof}

Note that the form \( \omega_I^N \) does not induce a Kähler form on
\( \CH(m+1) \).  Indeed direct computation gives \(
d_W\iota^*\omega_I^N = \alpha_{IJK}/\mu^2 \), which is non-zero.

\section{The hyperbolic plane}
\label{sec:hyperbolic-plane}

To provide more concrete examples of the c-map, we consider indefinite
projective special Kähler structures on open subsets~\( S \) of the
hyperbolic plane \( \RH(2) \) with constant curvature Kähler metric.
The space~\( \RH(2) \) is a solvable group of dimension~\( 2 \) with
one-dimensional derived algebra.  Choosing any unit vector~\( B \) in
the derived algebra and putting \( A = -IB \), we obtain an
orthonormal basis for the Lie algebra satisfying
\begin{equation*}
  [A,B] = \lambda B
\end{equation*}
for some non-zero constant \( \lambda \in \bR \).

Write \( a \), \( b \) for the dual basis to \( A,B \).  Then these
one-forms satisfy
\begin{equation*}
  da = 0\quad\text{and}\quad db = - \lambda a\wedge b.
\end{equation*}
The metric is \( g_S = a^2 + b^2 \) and the corresponding Kähler form
is \( \omega_S = a \wedge b \).  We note that \( Ia = b \).

Locally any special Kähler cone \( C \) over \( S \subset \RH(2) \) is
of the form \( C = {\bR_{>0} \times C_0} \).  Here \( C_0 \) is a
bundle over~\( S \), with connection one-form~\( \varphi \) satisfying
\( d\varphi = 2\pi^*\omega_S = 2\tilde a \wedge \tilde b \), where \(
\tilde a = \pi^* a \), etc., corresponding to reduction at level \(
g(X,X) = -1 \) in~\eqref{eq:curv-C}.

\begin{lemma}
  The metric and Kähler form of \( C = \bR_{>0} \times C_0 \) are
  \begin{equation*}
    g_C = \hat a^2 + \hat b^2 - \hat\psi^2 - \hat\varphi^2 ,\qquad
    \omega_C =  \hat a\wedge \hat b - \hat\varphi \wedge \hat\psi,
  \end{equation*}
  where \( \hat a = t\tilde a \), \( \hat b = t\tilde b \), \( \hat
  \varphi = t \varphi \) and \( \hat \psi = dt \), with \( t \)
  denoting the standard coordinate on \( \bR_{>0} \).  The conic
  symmetry \( X \) satisfies
  \begin{equation*}
    IX = t\frac\partial{\partial t}.
  \end{equation*}
\end{lemma}

\begin{proof}
  In general the metric is as claimed with \( \hat\psi = k\,dt \) for
  some \( k \in \bR_{>0} \) and \( \omega_C = \hat a\wedge \hat b -
  \varepsilon \hat\varphi \wedge \hat\psi\), for some \( \varepsilon =
  \pm1 \). For this structure to be Kähler we need \( d\omega_C = 0
  \).  However,
  \begin{equation}
    \label{eq:d-C}
    \begin{gathered}
      d\hat\psi = 0,\quad
      d\hat\varphi = \frac 1t(dt\wedge\hat\varphi
        + 2\hat a\wedge\hat b),\\  
      d\hat a = \frac 1t dt\wedge\hat a,\quad
      d\hat b = \frac1t(dt\wedge\hat b- \lambda\hat a\wedge\hat b),
    \end{gathered}
  \end{equation}
  so \( d\omega_C = 2\,\hat a\wedge \hat b \wedge (dt - \varepsilon
  \hat \psi ) / t \), implying that \( \varepsilon = +1 \) and \( k=1
  \).  Noting that \( I\hat a = \hat b \) and \( I\hat\varphi =
  \hat\psi \), we have from~\eqref{eq:d-C} that \( \Lambda^{1,0} =
  \operatorname{Span}\{\hat a - i\hat b,\hat \varphi - i \hat\psi\} \)
  satisfies \( d\Lambda^{1,0} \subset \Lambda^{2,0} + \Lambda^{1,1}
  \), and so \( I \) is integrable.

  To obtain the conic symmetry, write \( IX = f \partial/\partial t
  \).  Then \( L_{IX}\hat a = IX \hook d\hat a + d(IX\hook \hat a) =
  f\hat a/t \), and similarly for \( \hat b \) and \( \hat\varphi \).
  On the other hand \( L_{IX}I=0 \) and \( \hat\psi = I\hat\varphi \)
  give \( L_{IX}\hat\psi = IL_{IX}\hat\varphi = f\hat\psi/t \).  Thus
  the condition \( L_{IX}g = 2g \) implies that \( f = t \).
\end{proof}

\begin{lemma}
  \label{lem:Riem-flat}
  The pseudo-Riemannian metric \( g_C \) is flat if and only if \(
  \lambda^2=4 \).
\end{lemma}

\begin{proof}
  With respect to the unitary coframe \( s^*\theta = (\hat a,\hat
  b,\hat\varphi,\hat\psi) \) the connection one-form \( \oLC \) is
  uniquely determined by \( ds^*\theta = - s^*\oLC \wedge s^*\theta
  \), with \( \oLC^T\G + \G \oLC = 0 \) and \( \ii\oLC = \oLC\ii \).
  It follows from~\eqref{eq:d-C} that
  \begin{equation*}
    s^*\oLC =
    \frac1t
    \begin{pmatrix}
      0 & \hat\varphi + \lambda\hat b   & \hat b & \hat a \\
      -\hat\varphi - \lambda \hat b & 0 & -\hat a & \hat b \\
      \hat b  & -\hat a & 0 & \hat\varphi \\
      \hat a & \hat b & -\hat\varphi & 0
    \end{pmatrix}
    .
  \end{equation*}
  The curvature is then given by
  \begin{equation*}
    s^*\OLC = s^*(d\oLC + \oLC\wedge\oLC) = \frac{4-\lambda^2}{t^2}
    \begin{pmatrix}
      0 & \hat a\wedge\hat b&0 & 0 \\
      -\hat a\wedge\hat b &0&0 & 0 \\
      0 & 0 & 0 & 0\\
      0 & 0 & 0 & 0
    \end{pmatrix}
  \end{equation*}
  and the flatness result follows.
\end{proof}

\begin{proposition}
  \label{prop:H2-lambda}
  The cone \( (C,g_C,\omega_C) \) over \( S \subset \RH(2) \) is conic
  special Kähler if and only if \( \lambda^2 = 4/3 \) or~\( 4 \).
\end{proposition}

\begin{proof}
  We need to determine when the geometry admits a flat symplectic
  connection.  In the notation of the proof of Lemma~\ref{lem:Riem-flat},
  we need to determine a matrix-valued one-form \( \eta \) such that \(
  s^*\oN = s^*\oLC + \eta \) is flat, with
  \begin{enumerate*}
  \item\label{it:st} \( \eta \wedge s^*\theta = 0 \),
  \item\label{it:I-eta} \( \ii\eta = - \eta\ii \) and
  \item\label{it:eta-s} \( \eta^T\Gi = - \Gi\eta \).
  \end{enumerate*}
  Furthermore the analysis of Lemma~\ref{lem:conic-connection}, shows that
  we must have \( X \hook \eta = 0 \) and \( IX \hook \eta = 0 \).
  The latter implies that the entries of \( \eta \) lie in the span of
  \( \hat a \) and \( \hat b \).  Now \ref{it:I-eta} and
  \ref{it:eta-s} imply
  \begin{equation*}
    \eta =
    \begin{pmatrix}
      u & v & p & q \\
      v & -u & q & -p \\
      -p & -q & x & y \\
      -q & p & y & -x
    \end{pmatrix}
  \end{equation*}
  with all entries in the span of \( \hat a \) and \( \hat b \).
  Using \ref{it:st}, and considering the coefficients of \( \hat\varphi
  \) and \( \hat\psi \) gives first \( p = 0 = q \) and then \( x = 0
  = y \).  We then have
  \begin{equation}
    \label{eq:ab-uv}
    u \wedge \hat a + v \wedge \hat b = 0 \quad\text{and}\quad
    v \wedge \hat a - u \wedge \hat b = 0.
  \end{equation}
  The curvature of \( \nabla \) is \( \ON = d\oN + \oN \wedge \oN \)
  which pulls back to
  \begin{equation*}
    \begin{pmatrix}
      U & V + W   & 0 & 0\\
      V - W & - U & 0 & 0\\
      0 & 0 & 0 & 0 \\
      0 & 0 & 0 & 0 
    \end{pmatrix}
  \end{equation*}
  where
  \begin{gather*}
    U = du + \frac2t(\hat\varphi + \lambda\hat b)\wedge v,\quad
    V = dv - \frac2t(\hat\varphi + \lambda\hat b)\wedge u,\\
    W = \frac{4-\lambda^2}{t^2}\hat a\wedge\hat b + 2u\wedge v.
  \end{gather*}
  Equations~\eqref{eq:ab-uv} are solved in general by writing \( u =
  s_+\tilde a + s_-\tilde b \), \( v = s_-\tilde a - s_+\tilde b \),
  since \( \hat a=t\tilde a \) etc.  Then \( W = 0 \) is equivalent to
  \begin{equation}
    \label{eq:rs-lambda}
    s_+^2 + s_-^2 = \tfrac12 (4-\lambda^2)
  \end{equation}
  which is constant.  We may thus write \( s_+ = r\cos z \), \( s_-
  =r\sin z \) for some local function \( z \) and constant \( r>0 \)
  satisfying \( 2r^2 = (4-\lambda^2) \).

  Computing \( du \) and \( dv \), the equations \( U = 0 = V \) become
  \begin{equation*}
    (dz - 2\varphi - 3\lambda\tilde b) \wedge v = 0 = (dz - 2\varphi - 
    3\lambda\tilde b) \wedge u
  \end{equation*}
  If \( r=0 \), we have \( u=v=0 \), \( \NB = \NC \) and \(
  \lambda^2=4 \).  For \( r\ne 0\), we see that \( dz = 2\varphi +
  3\lambda\tilde b \).  In particular the right-hand side must be
  closed.  But \( d(2\varphi + 3\lambda\tilde b) = (4-3\lambda^2)\tilde
  a\wedge\tilde b \), so \( \lambda^2 = 4/3 \), as claimed.
\end{proof}

Note that this result does not require the conic symmetry to be
periodic or even quasi-regular.

\subsection{The flat case}
\label{sec:underst-twist-flat}

In the case \( \lambda = 2 \), we have that the Levi-Civita and
special Kähler connections coincide.  Let \( C_0 \) be the principal
\( S^1 \)-bundle over all of~\( \RH(2) \), with connection one form \(
\varphi \) satisfying \( d\varphi = 2\pi^*(a\wedge b) = 2\tilde
a\wedge\tilde b \).  Write the principal action as \( p \mapsto
e^{i\tau}\cdot p \).  The proof of Lemma~\ref{lem:Riem-flat} provides us
with the derivative of the unitary coframe~\( s^*\theta = (\hat a,\hat
b,\hat \varphi, \hat \psi) \).  The hyperKähler structure on \( H =
T^*C \) of the rigid c-map is
\begin{gather*}
  g_H = \hat a^2 + \hat b^2 -  \hat \varphi^2 - \hat \psi^2 + \hat A^2
  + \hat B^2 - \hat \Phi^2 - \hat \Psi^2,\\
  \omega_I = \hat a \wedge \hat b - \hat \varphi \wedge \hat \psi
  - \hat A \wedge \hat B + \hat \Phi \wedge \hat \Psi,\\
  \omega_J = \hat A\wedge\hat a + \hat B\wedge\hat b + \hat\Phi \wedge
  \hat\varphi + \hat\Psi\wedge\hat\psi,\\
  \omega_K = \hat A\wedge\hat b - \hat B\wedge\hat a + \hat\Phi \wedge
  \hat\psi - \hat\Psi \wedge \hat\varphi,
\end{gather*}
where \( (\hat A,\hat B,\hat \Phi,\hat \Psi) \coloneqq s^*\alpha = dx -
xs^*\oLC \).  We have \( ds^*\alpha = -s^*\alpha\wedge s^*\oLC \) and
\( \hX \hook s^*\alpha = 0 \) by definition.  Now \( \alpha_I = I\hX
\hook g_H = IX \hook (-\hat\psi^2) = -t\hat\psi \), so \( \alpha_0 =
-I\alpha_I = -t\hat\varphi \) and \( \mu = \tfrac12\norm{\hX}^2 = -t^2/2
\).  Similarly \( \alpha_J = I\hX \hook \omega_K = -t\hat\Phi \) and \(
\alpha_K = -I\hX \hook\omega_J = -t\hat\Psi = I\alpha_J \).

Now the metric we twist is 
\begin{equation*}
  \begin{split}
    g_N &= -\frac1\mu g_H +
    \frac1{\mu^2}g(\alpha_0^2+\alpha_I^2+\alpha_J^2+\alpha_K^2)\\
    &=\frac2{t^2}(\hat a^2 + \hat b^2 +  \hat \varphi^2 + \hat \psi^2 +
    \hat A^2 + \hat B^2 + \hat \Phi^2 + \hat \Psi^2),
  \end{split}
\end{equation*}
which is a complete metric on \( (t>0) \).  Thus \( g_N \) has
constant coefficients with respect to the coframe \( (\tilde a, \tilde
b, \varphi, \tilde\psi, \tilde A,\dots, \tilde \Psi) \), where the
last five terms are given by \( \tilde\psi = \hat\psi /t \), etc.

The twist data is
\begin{equation*}
  F = - \hat a \wedge \hat b + \hat \varphi \wedge \hat \psi
  - \hat A \wedge \hat B + \hat \Phi \wedge \hat \Psi
\end{equation*}
with twist function \( a \) equal to \( \mu = -t^2/2 \).

The coframe~\( \gamma \coloneqq s^*\theta/t = (\tilde a,\tilde
b,\varphi,\tilde\psi) \) on~\( C \) is invariant under~\( X \), so we
can compute these twisted differentials immediately:
\begin{equation}
  \label{eq:dWa-psi}
  \begin{gathered}
    d_W\tilde a = d\tilde a = 0,\quad
    d_W\tilde b = d\tilde b = 2\tilde b\wedge\tilde a,\\
    d_W\varphi = d\varphi +\frac2{t^2}F\wedge X\hook\varphi = 2\tilde
    a\wedge\tilde b + \frac2{t^2}F = 2(\varphi \wedge\tilde\psi - \tilde
    A\wedge\tilde B + \tilde\Phi\wedge\tilde\Psi),\\
    d_W\tilde\psi = d(dt/t) = 0.
  \end{gathered}
\end{equation}
On the other hand \( \tilde\delta \coloneqq (s^*\alpha)/t = (\tilde
A,\tilde B,\tilde\Phi,\tilde\Psi) \), has \( L_{\hX} \tilde\delta =
(L_{\hX}s^*\alpha) /t = \hX \hook ds^*\alpha / t =
\tilde\delta(\hX\hook s^*\oLC) = -\tilde\delta\ii \), so these forms
are not invariant.  However the bundle \( C_0 \to \RH(2) \) is
trivial; indeed \( d\varphi = 2\tilde a \wedge\tilde b = -d\tilde b
\), and we may write \( \varphi = d\tau - \tilde b \) with \( \exp(\ii
\tau) \) globally defined.  We have \( X\tau = 1 \), and \(
L_{\hX}(\tilde\delta e^{\ii \tau}) = 0 \).  Putting \( \delta =
\tilde\delta e^{\ii \tau} \), we find that
\begin{equation}
  \label{eq:dWdelta}
  \begin{split}
    d_W\delta&= d(\frac1ts^*\alpha\, e^{\ii \tau}) \\
    &= -\frac 1{t^2}\hat\psi \wedge s^*\alpha\, e^{\ii \tau} -
    \frac1ts^*\alpha \wedge s^*\oLC\, e^{\ii \tau} - \frac1ts^*\alpha
    \wedge \ii e^{\ii \tau} d\tau \\
    &= -\tilde\psi\wedge\delta - \delta\wedge s^*\oLC - \delta
    \wedge (\varphi + \tilde b)\ii\\ 
    &= \delta \wedge
    \begin{pmatrix}
      \tilde\psi&-\tilde b&-\tilde b&-\tilde a\\
      \tilde b&\tilde\psi&\tilde a&-\tilde b\\
      -\tilde b&\tilde a&\tilde\psi&\tilde b\\
      -\tilde a&-\tilde b&-\tilde b&\tilde\psi
    \end{pmatrix}
    ,
  \end{split}
\end{equation}
which has constant coefficients with respect to the coframe \(
(\gamma,\delta) \), as do the relations~\eqref{eq:dWa-psi}.

As \( g_N \) is complete, it follows that the universal cover of the
twist gives a left-invariant metric on a Lie group~\( G \).  Writing
\( \epsilon = \frac1{\sqrt2} ({\delta_1+\delta_4},\allowbreak
{\delta_2-\delta_3},\allowbreak {-\delta_2-\delta_3},\allowbreak
{-\delta_1+\delta_4}) \), equation~\eqref{eq:dWdelta} gives
\begin{equation*}
  d_W\epsilon = \epsilon\wedge
  \begin{pmatrix}
    \tilde\psi - \tilde a& 0 & 2\tilde b & 0 \\
    0 & \tilde\psi - \tilde a& 0 & -2\tilde b \\
    0 & 0 & \tilde\psi + \tilde a& 0 \\
    0 & 0 & 0 & \tilde\psi + \tilde a
  \end{pmatrix}
\end{equation*}
and we have \( d_W\varphi = 2(\varphi \wedge \tilde\psi +
\epsilon_1\wedge\epsilon_3 + \epsilon_4\wedge\epsilon_2) \).

The Lie algebra \( \lie g \) is completely solvable, with derived
algebra~\( \lie n \) of codimension~\( 2 \), while \( \lie n \) is \(
3 \)-step nilpotent, with \( \lie n' = [\lie n,\lie n] \) of dimension
\( 3 \) and \( \lie n^{(2)} = [\lie n,\lie n'] \) of dimension~\( 1
\).  We see that the Lie algebra is isomorphic to the solvable algebra
corresponding to the non-compact symmetric space \(
\operatorname{Gr}_2^+(\bC^{2,2}) = \Un(2,2)/(\Un(2)\times\Un(2)) \).
This solvable algebra may be identified with the Lie algebra of
matrices
\begin{equation*}
  \begin{pmatrix}
    x & u & v   & iw           \\
    0 & y & iz & -\overline v \\
    0 & 0 & -y  & -\overline u \\
    0 & 0 & 0   & - x
  \end{pmatrix},
  \qquad x,y,z,w \in \bR,\ u,v\in \bC,
\end{equation*}
when one realises \( \un(2,2) \) as matrices preserving \( \ii \) and
the inner product given by a matrix with non-zero entries \( 1 \) in
each anti-diagonal position \( (i,4-i) \).  Dually \( (\tilde a,\tilde
b,\varphi,\tilde \psi,\epsilon_1,\dots,\epsilon_4) \) correspond to \(
((y=1),(z=2),(w=1), (x=1),(\re u=1),(\im u=1),(\im v=1),(\re v=1)) \),
where for example \( (y=1) \) means the matrix with \( y=1 \) and all
other variables zero.

\subsection{The non-flat case}
\label{sec:non-flat-case}

In this case, we take \( \lambda = 2/\sqrt3 \).  The difference~\(
\eta \) between \( s^*\oN \) and \( s^*\oLC \) is described in the
proof of Proposition~\ref{prop:H2-lambda}.  Note that the local function~\( z \)
appearing there satisfies \( dz = 2(\varphi + \sqrt 3\, \tilde b) \),
so \( \varphi = \tfrac12dz - \sqrt 3\,\tilde b \).  But \( d\varphi =
2\tilde a\wedge\tilde b = -\sqrt 3d\tilde b \), so corresponding to
the flat case we write \( \tau = z/2 \).  This again has \( X\tau = 1
\).

If we write \( (\hat A,\hat B,\hat \Phi,\hat\Psi) = s^*\alpha = dx -
xs^*\oN = dx - xs^*\oLC - x\eta \), the description of \( g_H \) and
\( g_N \) is as in the flat case, and the twist data is the same.
Putting \( \gamma = (\tilde a,\tilde b,\varphi,\tilde\psi) \), we have
\begin{equation}
  \label{eq:dWa2-psi}
  \begin{gathered}
    d_W\tilde a = d\tilde a = 0,\quad
    d_W\tilde b = d\tilde b = - \frac2{\sqrt3}\tilde a\wedge\tilde b,\\
    d_W\varphi = 2(\varphi \wedge\tilde\psi - \tilde A\wedge\tilde B +
    \tilde\Phi\wedge\tilde\Psi),\quad
    d_W\tilde\psi = d(dt/t) = 0.
  \end{gathered}
\end{equation}
Once again \( \tilde\delta \coloneqq (s^*\alpha)/t = (\tilde A,\tilde
B,\tilde\Phi,\tilde\Psi) \), has \( L_{\hX} \tilde\delta =
(L_{\hX}s^*\alpha) /t = \hX \hook ds^*\alpha / t =
\tilde\delta(\hX\hook s^*\oLC) = -\tilde\delta\ii \), so we consider
\( \delta \coloneqq \tilde\delta e^{\ii \tau} \).
This satisfies
\begin{equation}
  \label{eq:dWdelta2}
  \begin{split}
    d_W\delta&= d(\frac1ts^*\alpha\, e^{\ii \tau}) \\
    &= -\frac 1{t^2}\hat\psi \wedge s^*\alpha\, e^{\ii \tau} -
    \frac1ts^*\alpha \wedge (s^*\oLC + \eta)\, e^{\ii \tau} -
    \frac1ts^*\alpha \wedge e^{\ii \tau}\ii d\tau \\
    &= - \tilde\psi\wedge\delta - \delta\wedge (s^*\oLC +
    e^{-\ii\tau}\eta e^{\ii \tau}) - \delta
    \wedge (\varphi + \sqrt 3\,\tilde b)\ii\\ 
    &= - \tilde\psi\wedge\delta - \delta\wedge (s^*\oLC +
    \eta e^{2\ii \tau}) - \delta
    \wedge (\varphi + \sqrt 3\,\tilde b)\ii\\ 
    &=\delta \wedge\left\{\tilde\psi 1_4 +
      \tilde a
      \begin{pmatrix}
        -2/\sqrt 3 & 0 & 0 & -1 \\
        0 & 2/\sqrt 3 & 1 & 0 \\
        0 & 1 & 0 & 0 \\
        -1 & 0 & 0 & 0
      \end{pmatrix}\vrule width 0pt height 7ex\right.
      \eqbreak[8]
      \left.\vrule width 0pt height 7ex
        + \tilde b
      \begin{pmatrix}
        0 & \sqrt 3 & -1 & 0 \\
        1/\sqrt 3 & 0 & 0 & -1 \\
        -1 & 0 & 0 & \sqrt 3 \\
        0 & -1 & -\sqrt 3 & 0
      \end{pmatrix}
    \right\}
    ,
  \end{split}
\end{equation}
which has constant coefficients with respect to the coframe \(
(\gamma,\delta) \).  As \( g_N \) is complete, we conclude that the
universal cover of the twist is a Lie group~\( G \), with
left-invariant quaternionic Kähler metric.

Let \( \zeta = \sqrt[4]3 \) and put \( \epsilon
= \frac1{\sqrt2}\bigl(\zeta(\delta_2\zeta+\delta_3/\zeta),\allowbreak
\frac1\zeta(\delta_1/\zeta-\delta_4\zeta),\allowbreak
\frac1\zeta(\delta_2/\zeta-\delta_3\zeta),\allowbreak
\zeta(\delta_1\zeta+\delta_4/\zeta)\bigr)
\).  We have
\begin{equation*}
  d_W\epsilon
  = \epsilon\wedge \left\{\tilde\psi 1_4+\frac1{\sqrt3}\tilde a
  \begin{pmatrix}
    3&0&0&0 \\
    0&1&0&0 \\
    0&0&-1&0 \\
    0&0&0&-3
  \end{pmatrix}
  +\frac2{\sqrt3}\tilde b
  \begin{pmatrix}
    0&0&0&0 \\
    3&0&0&0 \\
    0&2&0&0 \\
    0&0&1&0
  \end{pmatrix}
\right\}.
\end{equation*}
Moreover, \( d_W\varphi = 2\varphi\wedge\tilde\psi -
3\epsilon_2\wedge\epsilon_3 + \epsilon_1\wedge\epsilon_4 \).  The Lie
algebra is completely solvable with derived algebra \( \lie n \) of
codimension~\( 2 \), which is \( 4 \)-step nilpotent.  With \( a' =
\tilde a/\sqrt3 \) and \( b' = \frac2{\sqrt3}\tilde b \), we have \(
d_Wa' = 0 \), \( d_Wb' = 2b'\wedge a' \).  In this coframe \(
(a',b',\varphi,\tilde\psi,\epsilon_1,\dots,\epsilon_4) \) we see the
structure of solvable algebra associated to \( G_2^*/\SO(4) \): an
explicit matrix representation of this solvable algebra is provided by
Castrillón López, Gadea and Oubiña \cite{Castrillon-Lopez--GO:hqK8};
in their notation the dual to our coframe is \(
(A_2,U_3,X_3,2A_1+A_2,-X_2,U_1,U_2,X_1) \).

\section{Deformations and related geometries}
\label{sec:deform-relat-geom}

Our description of the c-map via the twist construction has the
advantage that there are obvious changes that can be made that still
yield quaternionic Kähler twists. 

\begin{asparaenum}
\item The most obvious is that a constant \( c \) can be subtracted
  from the~\( \mu \) given in \S \ref{sec:geometry-twist}, as in
  Theorem~\ref{thm:qK12}.  This change has been identified by
  Alexandrov et al.\ \cite{Alexandrov-DP:qK-hK} and Alekseevsky et
  al.\ \cite{Alekseevsky-CDM:qK-special} as the one-loop deformation
  of the Ferrara-Sabharwal metric, originally found by Cecotti,
  Ferrara and Girardello \cite{Cecotti-FG:II} and also studied by
  Robles Llana, Saueressig and Vandoren \cite{Robles--Llana-SV:loop}.
\item\label{it:tri-1} The rigid c-map produces hyperKähler metrics
  with many tri\hyphen holomorphic isometries.  If \( Y \) is one of
  these, we may consider twists by \( Z = \hX + \lambda Y \) for any
  real constant \( \lambda \).  Since \( Z \) satisfies \( L_Z\omega_J
  = \omega_K \) and is a Kähler isometry for \( (g,\omega_I) \),
  Theorem~\ref{thm:qK12} shows that the twists of \( H = T^*C \) by \(
  Z \) are still quaternionic Kähler.
\item\label{it:tri-2} If \( Y \) is a tri-Hamiltonian isometry acting
  freely and properly on \( H \) and commuting with \( \hX \), the one
  can create a new hyperKähler manifold with rotating circle action as
  follows.  Recall that the hyperKähler modification
  \cite{Dancer-S:mod} of \( H=T^*C \) by \( Y \) is \( H_\hmod = (H
  \times \bH) \hkq \bR \), the hyperKähler quotient of \( H\times \bH
  \) by the action \( (p,q) \mapsto (t\cdot p,e^{-it}q) \), where \( Y
  \) generates the action \( p \mapsto t\cdot p \).  On the product \(
  H \times \bH \) we have a rotating circle action given by the action
  of \( \hX \) on the first factor and that of \( q \mapsto
  e^{it/2}qe^{-it/2} \) on~\( \bH \).  This action descends to \(
  H_\hmod \) such that the hypotheses of Theorem~\ref{thm:qK12} are
  satisfied.  It is therefore possible to twist \( H_\hmod \) to
  produce quaternionic Kähler metrics.
\item\label{it:tri-3} The construction of item~\ref{it:tri-2} may be
  generalised further.  Firstly one may replace the factor \( \bH \)
  by any hyperKähler four-manifold that admits a tri-Hamiltonian
  action that commutes with a circle action rotation
  satisfying~\eqref{eq:rot-X}.  In general this is specified by \( S^1
  \)-invariant monopoles on open subsets \( U \) of \( \bR^3 \),
  cf.~Hitchin~\cite{Hitchin:Montreal}.  Such monopoles are specified
  by a harmonic function on~\( U \) that is required to be positive.
  Secondly, in \cite{Swann:twist-mod} it shown how such generalised
  modifications may be interpreted as twists via tri-holomorphic
  isometries, and there form part of a wider class of hyperKähler
  twists constructions, corresponding to a relaxation of the
  positivity condition on the harmonic function on \( U \subset \bR^3
  \).
\end{asparaenum}

In Alexandrov et al.\
\cite{Alexandrov-PSV:linear-qK,Alexandrov-PSV:linear-hK} various
deformations of hyperKähler and quaternionic Kähler metrics are
discussed.  It seems likely that some of the constructions above under
items~\ref{it:tri-1} to~\ref{it:tri-3} describe their deformations,
and that some simply provide isometric deformations of the
quaternionic Kähler metric.  However, at this stage the correspondence
between these constructions in not clear.  We believe
items~\ref{it:tri-1} to~\ref{it:tri-3} describe a wider class of
quaternionic Kähler metrics.

\providecommand{\bysame}{\leavevmode\hbox to3em{\hrulefill}\thinspace}
\providecommand{\MR}{\relax\ifhmode\unskip\space\fi MR }
\providecommand{\MRhref}[2]{%
  \href{http://www.ams.org/mathscinet-getitem?mr=#1}{#2}
}
\providecommand{\href}[2]{#2}

\bigskip
{\small
  \setlength{\parindent}{0pt} Oscar Macia

  Departamento de Geometria y Topologia, Facultad de Ciencias
  Matematicas, Universidad de Valencia, C.\ Dr Moliner, 50, (46100)
  Burjassot (Valencia), Spain

 \textit{E-mail}: \url{oscar.macia@uv.es}

  \smallskip Andrew Swann

  Department of Mathematics, Aarhus University, Ny Munkegade 118, Bldg
  1530, DK-8000 Aarhus C, Denmark.

  \textit{E-mail}: \url{swann@imf.au.dk}
  \par}

\end{document}